\documentclass[11pt,a4paper]{article}
\usepackage{authblk}
\usepackage[utf8]{inputenc}
\usepackage[T1]{fontenc}
\usepackage{graphicx}
\usepackage{amssymb}
\usepackage{amsmath, amsthm}
\usepackage{amsfonts}
\usepackage[numbers,sort&compress]{natbib}
\usepackage{hyperref}
\usepackage{caption}
\usepackage{xcolor}
\newtheorem{theorem}{Theorem}

\numberwithin{deff}{section}

\title{Pantograph Difference Equations: Stability, Multistability, and Chaos}
\author[1]{Prashant M. Gade}
\author[2]{Sachin Bhalekar\footnote{Corresponding author: sachinbhalekar@uohyd.ac.in}}
\affil[1]{Ramniranjan Jhunjhunwala College of Arts, Science and Commerce, Ghatkopar West, Mumbai 400086, India}
\affil[2]{School of Mathematics and Statistics, University of Hyderabad, Hyderabad 500046, India}

\date{}
\begin{document}
	
	\maketitle

		\begin{abstract}
Pantograph equations are functional differential equations with proportional delay, typically involving terms of the form $x(at)$, where $0<a<1$. These equations have been studied for several decades and have applications in areas ranging from number theory to electrodynamics. In this work, we investigate pantograph difference equations as discrete dynamical maps of the form
\[
x_{i+1}=F(x_{i})+G(x_{h(i)}), \qquad i\geq 0,
\]
where $h(i)\leq i$. In particular, we consider $h(i)=\lfloor ai \rfloor$ and $h(i)=\lfloor ai^{\alpha}\rfloor$, with $0<a<1$ and $0<\alpha<1$. We distinguish between homogeneous pantograph difference equations, for which $F(x)=0$, and inhomogeneous pantograph difference equations, for which $F(x)\ne 0$. We first analyze linear pantograph difference equations and establish sufficient conditions for asymptotic stability and for nonconvergence to zero. For proportional delays, analytical stability and instability regions are combined with multi-horizon parameter scans to construct regime diagrams in the $(\lambda,\delta)$-plane. 
We then investigate pantograph versions of the logistic and tent maps. In the homogeneous case, the dynamics evolves through step-like plateaus whose lengths increase at least exponentially, leading to slow convergence and finite-time multistability within computationally accessible time scales. In the inhomogeneous case, numerical experiments show finite-time exponential separation of nearby trajectories, consistent with chaotic dynamics.
	\end{abstract}
	
	 
\section{Introduction}

Pantograph equations are functional differential equations with proportional delay, originating in models of current collection systems and subsequently developed as a class of functional differential equations with rescaled arguments \cite{ockendon1971dynamics,kato1971functional,iserles1993generalized}. A typical example has the form
\[
\dot x(t)=f(t,x(at)), \qquad 0<a<1,
\]
where the rate of change at $t$ depends on the state at the scaled argument $at$. This feature distinguishes pantograph equations from ordinary delay differential equations with constant delays and gives rise to a memory structure that changes with time \cite{bellen2003numerical}.

The original motivation came from electric locomotive pantographs, where reliable contact with overhead wires is essential for current collection \cite{ockendon1971dynamics}. Since then, pantograph-type equations have appeared in several areas, including electrodynamics, number theory, quantum mechanics, statistical mechanics, viscoelasticity, thermoelasticity, diffraction theory, and population dynamics. The proportional delay produces mathematical difficulties that are different from those in constant-delay systems, and it often leads to solutions with nonstandard asymptotic and stability properties \cite{kato1971functional,iserles1993generalized,buhmann1993stability}.

A substantial part of the literature has focused on analytical representations of solutions and their qualitative properties. Series-based methods, decomposition techniques, and related semi-analytical approaches have been developed for linear, nonlinear, commensurate, and incommensurate proportional delays. In particular, Bhalekar and Patade studied analytical solutions of nonlinear equations with proportional delays \cite{bhalekar2016analytical}, series solutions of the pantograph equation and their properties \cite{bhalekar2017series}, and pantograph equations with incommensurate delay \cite{patade2017analytical}. More recently, Bhalekar also investigated stability questions for pantograph delay differential equations \cite{bhalekar2026stability}. These works emphasize that proportional delay can change not only the form of solutions but also their convergence, stability, and long-time behavior.

Numerical analysis of pantograph equations has also been extensively studied. Discretization schemes, stability-preserving numerical methods, perturbation--iteration algorithms, Laplace-transform approaches, predictor--corrector methods, frame-based methods, and Runge--Kutta-type approaches have been developed for pantograph and pantograph-type delay differential equations \cite{buhmann1993stability,bellen2003numerical,bahsi2015numerical,jafari2021new,rihan2022theta,alrebdi2023accurate,jamaludin2025adaptive,mohammad2026carleson,rihan2024runge}. Recent studies further continue to develop analytical, series-based, numerical, fractional-order, controllability, and optimal-control treatments for homogeneous, inhomogeneous, and generalized pantograph systems \cite{albidah2023exact,aljohani2024analytical,ahmad2023controllability,yi2025novel,abusalim2024some}. The original engineering motivation also remains active, as recent work continues to analyze pantograph--overhead line interaction in railway systems \cite{bruni2025recent}.

There is also an established literature on the stability and asymptotics of discrete pantograph equations. Buhmann and Iserles \cite{buhmann1993stability} studied trapezoidal-rule discretizations of generalized pantograph differential equations, including neutral terms and interpolated delayed values. They obtained sufficient conditions for boundedness and algebraic decay, and also considered a broader class of linear recurrence relations with proportional delays. Their work establishes that slow algebraic relaxation and stability analysis of discrete proportional-delay systems predate the present investigation. The interpolated and neutral discretizations they analyze should nevertheless be distinguished from the single-delayed-value maps considered below.

P\'eics \cite{peics2000asymptotic} investigated the variable-coefficient equation $x_{n+1}=(1-a_n)x_n+b_nx_{p_n}$, with $0<a_n<1$, $p_n\leq n$, and $p_n\to\infty$. Weighted comparison estimates yield upper bounds and, under additional positivity assumptions, lower bounds for solutions. In particular, Corollaries 1 and 3 give algebraic and logarithmic decay bounds for $p_n=\lfloor n/p\rfloor$ and $p_n=\lfloor n^{1/p}\rfloor$, respectively, where $p>1$ is an integer, under suitable strict coefficient inequalities. For constant coefficients, the identification $1-a_n=\lambda$ and $b_n=\delta$ includes our scalar linear equation when $0<\lambda<1$. Thus, both proportional and sublinear delays, and the possibility of very slow convergence, already occur in the theory of linear pantograph difference equations.

An especially direct connection is provided by Kundr\'at \cite{kundrat2006asymptotics}, who studied $\Delta x_n=Ax_n+Bx_{\lfloor qn\rfloor}$, with $0<q<1$, and obtained polynomial upper bounds for systems with diagonal instantaneous coefficient matrices, with an extension to diagonalizable matrices. In the scalar case, setting $A=\lambda-1$, $B=\delta$, and $q=a$ gives exactly the proportional-delay version of our linear recurrence. For $|\lambda|<1$ and $\delta\ne0$, his Theorem 3.1 yields
\[
x_n=O(n^\gamma),\qquad
\gamma=\frac{\log\!\left(|\delta|/(1-|\lambda|)\right)}{\log(1/a)}.
\]
Consequently, $|\lambda|+|\delta|<1$ implies decay in this setting, so the proportional-delay case of Theorem~\ref{thm1} below recovers a known sufficient stability condition. The general-delay proof given here does not require a proportional or monotone delay. A positive value of $\gamma$ in Kundr\'at's upper bound does not, by itself, imply unboundedness. His discussion also illustrates misleading transient decay before later growth in the continuous pantograph equation, reinforcing the need for caution when interpreting finite-time simulations.

Against this background, we study delayed maps of the form
\[
x_{i+1}=F(x_{i})+G(x_{h(i)}), \qquad i\geq 0,
\]
where $h(i)$ is an integer-valued delay function. We focus mainly on proportional and sublinear proportional delays of the form $h(i)=\lfloor ai \rfloor$ and $h(i)=\lfloor ai^{\alpha}\rfloor$, where $0<a<1$ and $0<\alpha<1$. When $F=0$, we refer to the system as a homogeneous pantograph difference equation; when $F\ne 0$, we call it an inhomogeneous pantograph difference equation.

{Building on this established linear theory, our objective is to understand how proportional delay changes the behavior of discrete dynamical systems. We present elementary sufficient conditions for asymptotic stability and nonconvergence to zero, together with sign-dependent instability criteria and multi-horizon parameter scans. These results provide a foundation for the nonlinear case, where proportional delay can strongly alter convergence rates, transient behavior, and sensitivity to initial conditions.}

{The nonlinear focus distinguishes the present investigation from the three linear studies discussed above \cite{buhmann1993stability,peics2000asymptotic,kundrat2006asymptotics}. We examine pantograph versions of the logistic and tent maps, including their exact composition structure in the homogeneous case, increasing stasis intervals, and initial-condition dependence over finite observation windows. In the inhomogeneous case, we investigate finite-time exponential separation of nearby trajectories. These observations are interpreted as finite-time multistability and evidence consistent with chaotic dynamics, rather than as proofs of distinct asymptotic attractors or positive asymptotic Lyapunov exponents.}

{The remainder of this paper is organized as follows. Section \ref{sec2} develops the theory of linear pantograph difference equations, including the form of solutions, delay-independent stability criteria, and a numerical and analytical framework for the unclassified regions of the parameter plane. It also extends the discussion to nonlinear delayed maps and provides sufficient conditions for local stability of equilibrium points. Section \ref{sec3} studies pantograph versions of the logistic and tent maps and illustrates the resulting hierarchical dynamics, multistability, and chaotic behavior through numerical simulations. Conclusions, discussion, and future directions are given in Section \ref{sec4}.}

\section{Delay Difference Equations} \label{sec2}

We begin with the linear delay difference equation
\begin{equation}
x_{n+1} = \lambda x_n + \delta x_{h(n)}, \quad n\in \mathbf{N}_0 =\{0,1,2,\dots\},   \label{pmap} 
\end{equation}
where $\lambda,\delta\in\mathbb{R}$ and $h:\mathbf{N}_0\longrightarrow\mathbf{N}_0$ is a time-dependent delay function satisfying $h(n)\leq n$.
When $h(n)=\lfloor an\rfloor$, with $0<a<1$, the delay is proportional and equation (\ref{pmap}) is a pantograph difference equation. Here $\lfloor\cdot\rfloor$ denotes the integer-part function. We also consider the sublinear proportional delay $h(n)=\lfloor an^\alpha\rfloor$, where $0<a<1$ and $0<\alpha<1$.

\subsection{Structure of the solution}
We first describe the structure of the solution of (\ref{pmap}). The standard solution formula for the nonhomogeneous first-order difference equation
\begin{equation}
x_{n+1} = \lambda x_n + g(n),
\end{equation}
is given by \cite[page 4]{elaydi2005introduction}
\begin{equation}
x_n=\lambda^n x_0 + \sum_{r=0}^{n-1}\lambda^{n-r-1}g(r).
\end{equation}
Setting $g(n)=\delta x_{h(n)}$ and, for simplicity, taking $x_0=1$, the solution of (\ref{pmap}) can be written as
\begin{equation}
x_n=\lambda^n  + \delta \sum_{r_{0}=0}^{n-1}\lambda^{n-r_0-1}x_{h(r_0)}.\label{gsol}
\end{equation}
Repeated substitution in (\ref{gsol}) yields
\begin{equation}
x_n=\lambda^n + \delta \sum_{r_{0}=0}^{n-1}\lambda^{n-r_0-1+h(r_0)} +\delta^2 \sum_{r_{0}=0}^{n-1}\sum_{r_{1}=0}^{h(r_0)-1} \lambda^{n-r_0-r_1-2+h(r_0)} x_{h(r_1)}.
\end{equation}
{Continuing this process generates a decreasing sequence of indices $n>r_0>r_1>r_2>\cdots>r_k$ that terminates when $h(r_k)=0$. Since $x_0=1$, the solution can therefore be expressed in the form
\begin{equation}
x_n=\lambda^n + \delta P_0(\lambda) + \delta^2 P_1(\lambda)+\cdots + \delta^{k+1} P_k(\lambda).
\end{equation}
where each $P_j(\lambda)$ is a polynomial in $\lambda$ with nonnegative integer coefficients.}

Consequently, $x_n$ is a polynomial of total degree $n$ in the parameters $\lambda$ and $\delta$. The term $\lambda^n$ is always present, whereas the highest power of $\delta$ may be much smaller than $n$ because the delayed indices decrease according to $h$. This slow growth of the powers of $\delta$ is one of the characteristic effects of proportional delay. Several explicit solution terms for $h(n)=\lfloor n/m\rfloor$, $m=2,3,4,5$, are listed in the Appendix.

\subsection{Stability Analysis}
\subsubsection{Delay-independent stability conditions}
\begingroup
Every real number is an equilibrium of (\ref{pmap}) precisely when $\lambda+\delta=1$. If $\lambda+\delta\ne1$, the only equilibrium is zero.

The following theorem gives a sufficient condition for stability that is independent of the particular choice of the delay function $h$.
\begin{theorem}\label{thm1}
Suppose that $h(n)\to\infty$ as $n\to\infty$. The zero solution of the linear delay difference equation (\ref{pmap}) is asymptotically stable if $|\lambda| + |\delta| < 1$.
\end{theorem}
\textbf{Proof:}
Set $q=|\lambda|+|\delta|<1$. Induction gives $|x_n|\leq |x_0|$ for every $n$, so the solution is bounded and the zero solution is stable. Let
\[
L=\limsup_{n\to\infty}|x_n|.
\]
Because $h(n)\to\infty$, we also have $\limsup_{n\to\infty}|x_{h(n)}|\leq L$. Taking the upper limit in
\[
|x_{n+1}|\leq |\lambda|\,|x_n|+|\delta|\,|x_{h(n)}|
\]
therefore yields $L\leq qL$. Since $q<1$, it follows that $L=0$, and hence $x_n\to0$.
\qed

\textbf{Note}: The assumption $h(n)\to\infty$ is essential. For example, if $h(n)=0$ and $\delta\ne0$, the persistent term $\delta x_0$ generally prevents convergence to zero. The condition $|\lambda|+|\delta|<1$ is sufficient but not necessary for proportional delays. Figure \ref{fig02} shows a numerically decaying orbit of (\ref{pmap}) with $\lambda=-0.50$, $\delta=0.52$, and $h(n)=\lfloor 0.2n\rfloor$, even though $|\lambda|+|\delta|>1$.

\begin{figure}[h]
    \centering
    \includegraphics[width=\linewidth]{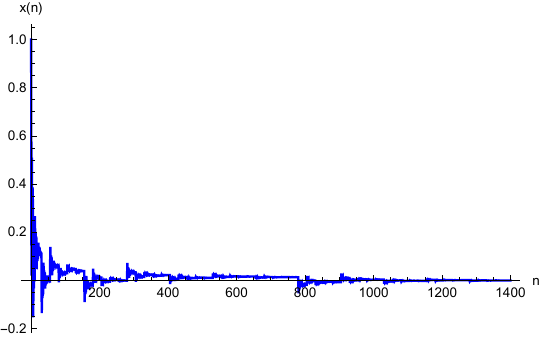}
    \caption{Numerically decaying orbit of (\ref{pmap}) for $h(n)=\lfloor0.2n\rfloor$, $\lambda=-0.50$, and $\delta=0.52$, outside the sufficient region $|\lambda|+|\delta|<1$.}
    \label{fig02}
\end{figure}

We next give a sufficient condition for instability in the first quadrant.
\begin{theorem} \label{thm2}
If $\lambda,\delta\geq0$ and $\lambda+\delta>1$, then the zero solution of (\ref{pmap}) is not asymptotically stable for any delay function satisfying $h(n)\leq n$.
\end{theorem}
\textbf{Proof:} Choose $x_0>0$. If $x_j\geq x_0$ for $0\leq j\leq n$, then
\[
x_{n+1}=\lambda x_n+\delta x_{h(n)}\geq(\lambda+\delta)x_0>x_0.
\]
The assertion follows by induction: the solution cannot converge to zero.
\qed
\bigskip

\begin{theorem} \label{thm5}
Under the hypotheses of Theorem \ref{thm2}, if $\lambda>0$, $\delta>0$, and $h$ is monotone non-decreasing, then the solution with $x_0>0$ is strictly increasing.
\end{theorem}
\textbf{Proof:} Assume $x_0>0$. Since $\lambda>0$ and $\delta>0$, all subsequent iterates are positive.

We show by induction that $\{x_n\}$ is strictly increasing and hence does not converge to zero.

Indeed, since $h(0)=0$,
\begin{equation*}
x_1=\lambda x_0 + \delta x_{h(0)}=  \lambda x_0 + \delta x_0 > x_0.  
\end{equation*}

Assume that $x_n>x_{n-1}$ for $n=1,2,\cdots,k$. Since $h$ is monotone non-decreasing, $h(k)\geq h(k-1)$. Therefore,

\begin{equation*}
    x_{k+1}=\lambda x_k + \delta x_{h(k)} > \lambda x_{k-1} + \delta x_{h(k-1)}=x_k.
\end{equation*}
Thus, $x_n$ is strictly increasing by induction.
\qed
\bigskip

The monotonicity assumption in Theorem \ref{thm5} is needed only for strict monotonic growth, not for the nonconvergence conclusion of Theorem \ref{thm2}. 

For the pantograph delay $h(n)=\lfloor an\rfloor$, the increasing solution in Theorem \ref{thm5} is in fact unbounded. Indeed, if it were bounded, it would converge to a limit $L\geq x_0>0$. Since $h(n)\to\infty$, we would also have $x_{h(n)}\to L$, and passage to the limit in (\ref{pmap}) would give $L=(\lambda+\delta)L$, contradicting $\lambda+\delta>1$. Thus, in the first quadrant, $\lambda+\delta<1$ gives asymptotic decay by Theorem \ref{thm1}, $\lambda+\delta=1$ gives the constant solution $x_n=x_0$, and $\lambda+\delta>1$ gives unbounded growth for $x_0>0$.

Now, we propose a result providing a few more instability regions in the parameter plane.
\begin{theorem}\label{thm3}
If
\[
|\lambda|>1+|\delta|,
\]
then the zero solution of (\ref{pmap}) is unstable for every delay function satisfying $h(n)\leq n$. More precisely, every solution with $x_0\ne0$ is unbounded.
\end{theorem}
\textbf{Proof:}
Set $q=|\lambda|-|\delta|>1$. We prove inductively that
\[
|x_{n+1}|\geq q|x_n|>|x_n|.
\]
Since $h(0)=0$, the reverse triangle inequality gives
\[
|x_1|=|\lambda+\delta|\,|x_0|
\geq (|\lambda|-|\delta|)|x_0|=q|x_0|.
\]
Now suppose that $|x_0|<|x_1|<\cdots<|x_n|$. Because $h(n)\leq n$, we have $|x_{h(n)}|\leq |x_n|$. Hence
\[
\begin{aligned}
|x_{n+1}|
&=|\lambda x_n+\delta x_{h(n)}|\\
&\geq |\lambda|\,|x_n|-|\delta|\,|x_{h(n)}|\\
&\geq (|\lambda|-|\delta|)|x_n|
=q|x_n|.
\end{aligned}
\]
Thus $|x_n|\geq q^n|x_0|\to\infty$, proving instability.
\qed
\bigskip

The next results adds unstable regions for the particular case $h(n)=\lfloor an\rfloor$.
\begin{theorem}\label{thm4}
Let $h(n)=\lfloor an\rfloor$, where $0<a<1$. If
\[
|\delta|>1+|\lambda|,
\]
then the zero solution of (\ref{pmap}) is unstable. More precisely, every solution with $x_0\ne0$ has an unbounded subsequence.
\end{theorem}
\textbf{Proof:}
Set
\[
c=\frac{|\delta|}{1+|\lambda|}>1.
\]
For every integer $k\geq0$, the interval
\[
\left[\frac{k}{a},\frac{k+1}{a}\right)
\]
has length $1/a>1$. It therefore contains an integer $n>k$ satisfying $\lfloor an\rfloor=k$.

Starting with $k_0=0$, suppose that $k_j$ has been chosen, and select $n_j>k_j$ such that $\lfloor an_j\rfloor=k_j$. From (\ref{pmap}),
\[
\delta x_{k_j}=x_{n_j+1}-\lambda x_{n_j}.
\]
Consequently,
\[
\begin{aligned}
|\delta|\,|x_{k_j}|
&\leq |x_{n_j+1}|+|\lambda|\,|x_{n_j}|\\
&\leq(1+|\lambda|)
\max\{|x_{n_j}|,|x_{n_j+1}|\}.
\end{aligned}
\]
Choose $k_{j+1}\in\{n_j,n_j+1\}$ so that
\[
|x_{k_{j+1}}|=\max\{|x_{n_j}|,|x_{n_j+1}|\}.
\]
Then $k_{j+1}>k_j$ and
\[
|x_{k_{j+1}}|\geq c|x_{k_j}|.
\]
It follows inductively that
\[
|x_{k_j}|\geq c^j|x_0|\longrightarrow\infty.
\]
Thus the solution has an unbounded subsequence, and the zero solution is unstable.
\qed
\bigskip

Theorems \ref{thm3} and \ref{thm4} also give rigorous instability regions outside the first quadrant. In the second quadrant, where $\lambda<0$ and $\delta>0$, the regions
\[
\delta>1-\lambda
\qquad\text{and}\qquad
\delta<-1-\lambda
\]
are unstable. The first inequality follows from Theorem \ref{thm4}, whereas the second follows from Theorem \ref{thm3}. In the third quadrant, where $\lambda<0$ and $\delta<0$, the regions
\[
\delta>1+\lambda
\qquad\text{and}\qquad
\delta<\lambda-1
\]
are unstable by Theorems \ref{thm3} and \ref{thm4}, respectively. In particular, Theorem \ref{thm4} proves instability throughout $\lambda\leq0$, $\delta<\lambda-1$, for proportional delays.
\bigskip

Outside the sufficient stability region, behavior can depend on the signs of the coefficients and on the delay. 

\subsubsection{Numerical experiments for the unclassified regions in the parameter plane}
The observations in the regions of the $(\lambda,\delta)$-parameter plane which are not covered by results in the previous subsection require particular care because pantograph delays can produce extremely long transients. A trajectory that appears bounded or decaying over a fixed time interval need not have reached its asymptotic regime. We therefore distinguish three types of behavior and provide a color code for the corresponding parameter values:
\begin{enumerate}
\item asymptotic decay, $x_n\to0$ with color code yellow;
\item bounded nonconvergent behavior, for which $\sup_n|x_n|<\infty$ but $x_n\not\to0$ with color code blue;
\item unbounded behavior, for which $\sup_n|x_n|=\infty$ with color code red.
\end{enumerate}
Because (\ref{pmap}) is linear and homogeneous in the initial value, it is sufficient to take $x_0=1$ when determining these alternatives; every other solution is obtained by scalar multiplication.

\paragraph{Finite-time parameter scan.}
To obtain a common numerical picture of the three quadrants, we used $h(n)=\lfloor an\rfloor$ with
\[
a\in\{0.1,0.2,\ldots,0.9\}.
\]
For each value of $a$, the $(\lambda,\delta)$-plane was sampled on a square mesh of spacing $0.1$ in $[-10,10]^2$. In the second, third, and fourth quadrants, only points which are not covered by the Theorems \ref{thm1}, \ref{thm2}, \ref{thm5}, \ref{thm3} and \ref{thm4} required numerical classification. At every sampled point, equation (\ref{pmap}) was iterated from $x_0=1$ up to $N=50000$. The following operational rules were used:
\begin{enumerate}
\item the point was classified as \emph{threshold-growing} if $|x_n|>10^4$ for some $n\leq N$, in which case the iteration was stopped and the point was saved in a list with color code red;
\item if the threshold was never crossed, the point was classified as \emph{numerically decaying} when
\[
\max_{N-9\leq n\leq N}|x_n|<10^{-2}
\]
and the point was saved in a list with color code yellow;
\item every remaining point was classified as \emph{finite-time bounded and nondecaying} and saved in a list with color code blue.
\end{enumerate}

\paragraph{Rescan of the initially blue region.}
The preceding finite-time criterion can incorrectly place an ultra-slow growing orbit in the bounded class. We therefore rescanned every initially blue mesh point. First, the line $\lambda+\delta=1$ was treated analytically rather than numerically, because $x_n\equiv x_0$ is an exact solution on this line; direct floating-point iteration may incorrectly destabilize this solution through amplified roundoff errors. Every other blue point was iterated to $N=10^6$. At the horizons
\[
H\in\{5\times10^4,10^5,2\times10^5,5\times10^5,10^6\},
\]
we recorded the half-window envelope
\[
E(H)=\max_{H/2<n\leq H}|x_n|
\]
and the root-mean-square amplitude over the same window,
\[
R(H)=\left(\frac{2}{H}\sum_{H/2<n\leq H}|x_n|^2\right)^{1/2}.
\]
The envelope $E(H)$ is a windowed maximum-norm measure, whereas $R(H)$ is the corresponding RMS measure of the typical signal magnitude and is less sensitive to an isolated spike; maximum-norm and RMS summaries are standard tools in numerical and random-data analysis \cite{press2007numerical,bendat2010random}. For either amplitude $A\in\{E,R\}$, the log--log slope between two horizons $H_i<H_j$ was computed as
\[
s_A(H_i,H_j)=\frac{\log\!\left(A(H_j)/A(H_i)\right)}{\log(H_j/H_i)}.
\]
Thus, if $A(H)$ behaves approximately as $CH^\beta$, then $s_A\approx\beta$: a positive value indicates increasing amplitude, a negative value indicates decay, and a value near zero indicates bounded or exceptionally slow behavior. This interpretation follows the standard use of slopes on logarithmic axes to estimate effective power-law exponents \cite{clauset2009powerlaw}. As emphasized in the power-law literature, however, a positive log--log slope by itself is not a goodness-of-fit test or a proof of asymptotic growth; here it is used only as a finite-horizon diagnostic. In the scan to $10^6$, growth was called persistent when the envelope increased by more than $1\%$ on at least two of the four successive horizon pairs and, in addition, either
\[
\frac{E(10^6)}{E(5\times10^4)}>1.05,
\qquad s_E(5\times10^4,10^6)>0.015,
\]
or the same two inequalities held with $E$ replaced by $R$. A point was also assigned to the growing class when the envelope ratio exceeded $1.10$ and its corresponding slope exceeded $0.03$. Points whose trend remained ambiguous were extended to $N=5\times10^6$. For this extended scan, the endpoint horizons $10^6$ and $5\times10^6$ were compared, and growth was recorded when the envelope or RMS ratio exceeded $1.05$ and its corresponding slope exceeded $0.015$. Independently of these slope criteria, crossing $|x_n|=10^4$ was taken as direct numerical evidence of growth. These finite-horizon tests provide numerical evidence rather than a proof of asymptotic instability. A decreasing envelope was assigned to the decaying class, while a point with no sustained upward or downward trend was retained only as a bounded candidate.



In Figure \ref{fig:linear-regime-diagrams}, the solid black diamond is the theoretical boundary $|\lambda|+|\delta|=1$ from Theorem \ref{thm1}; its strict interior is colored yellow analytically and was not numerically scanned. Solid red indicates the analytically proved unstable regions from Theorems \ref{thm2}--\ref{thm4}. Outside the analytically classified regions, colors are shown only at sampled mesh points. Small red squares indicate numerical instability or growth, including points in the extended plotting region whose trajectories cross $10^4$ by $N=50000$; yellow squares indicate numerical decay; and blue squares indicate either an exact constant solution or a bounded candidate retained after the multi-horizon rescan to as many as $5\times10^6$ iterates. 

\begin{center}
    \includegraphics[width=0.48\textwidth]{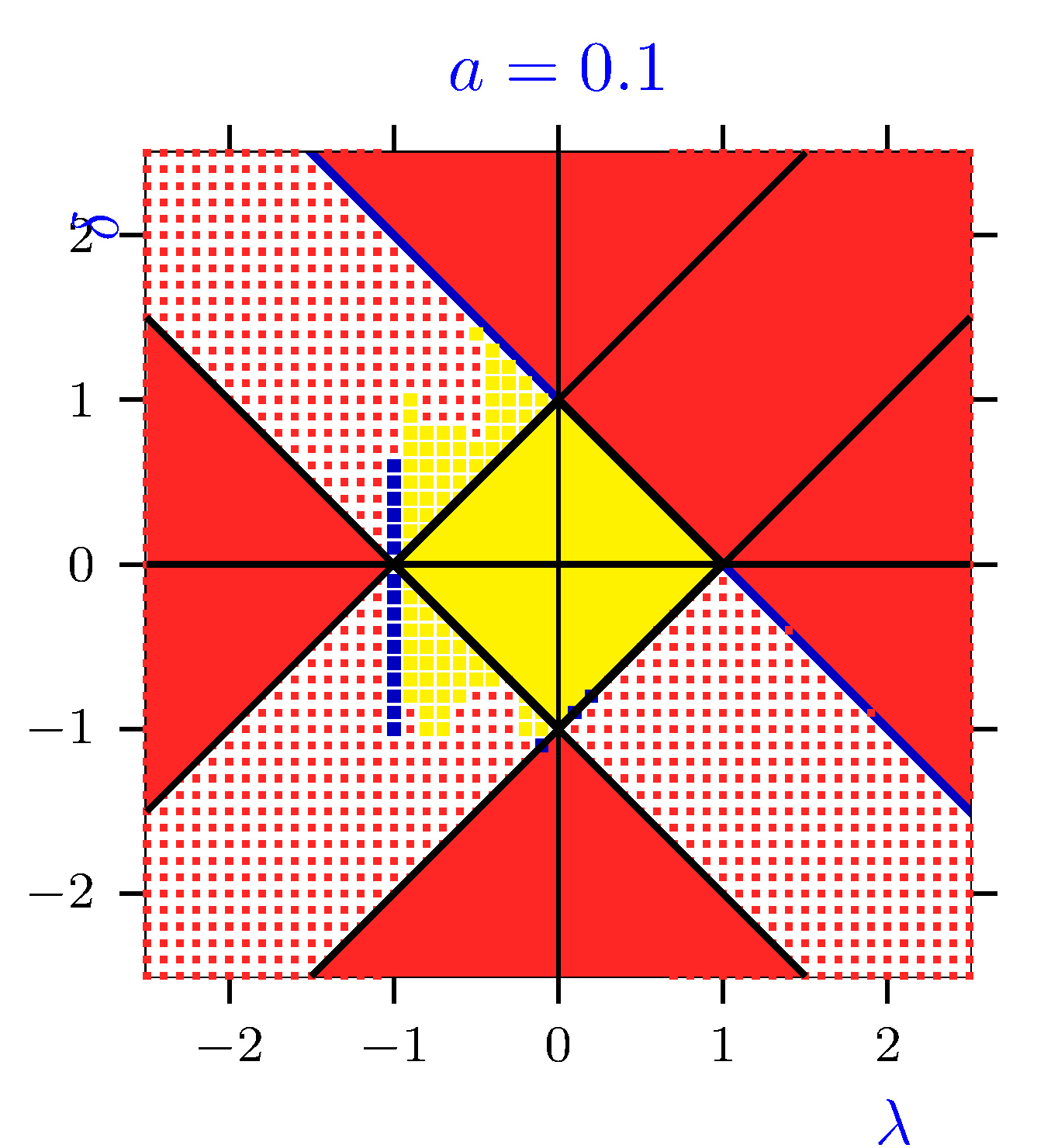}\hfill
    \includegraphics[width=0.48\textwidth]{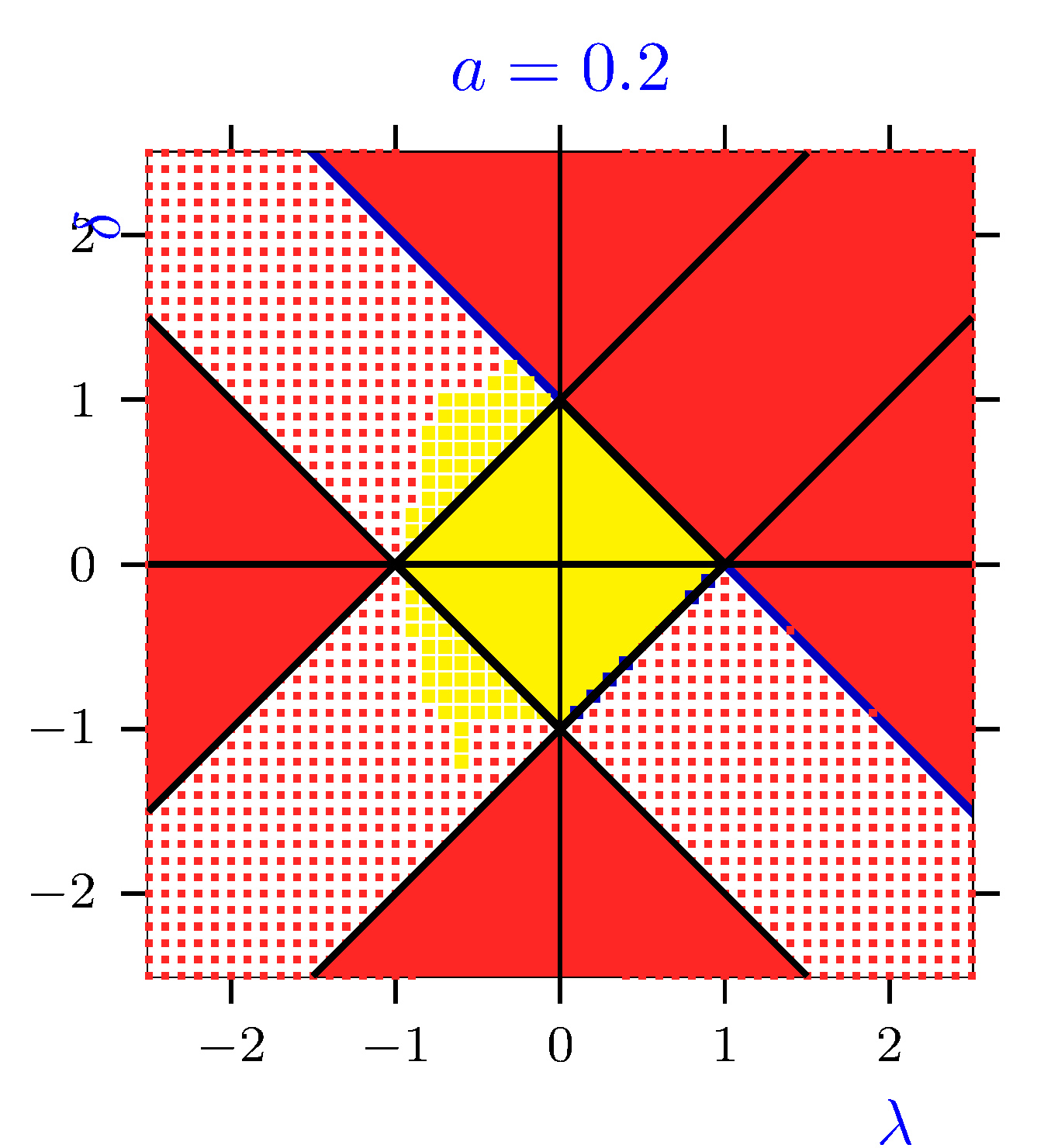}

    \medskip
    \includegraphics[width=0.48\textwidth]{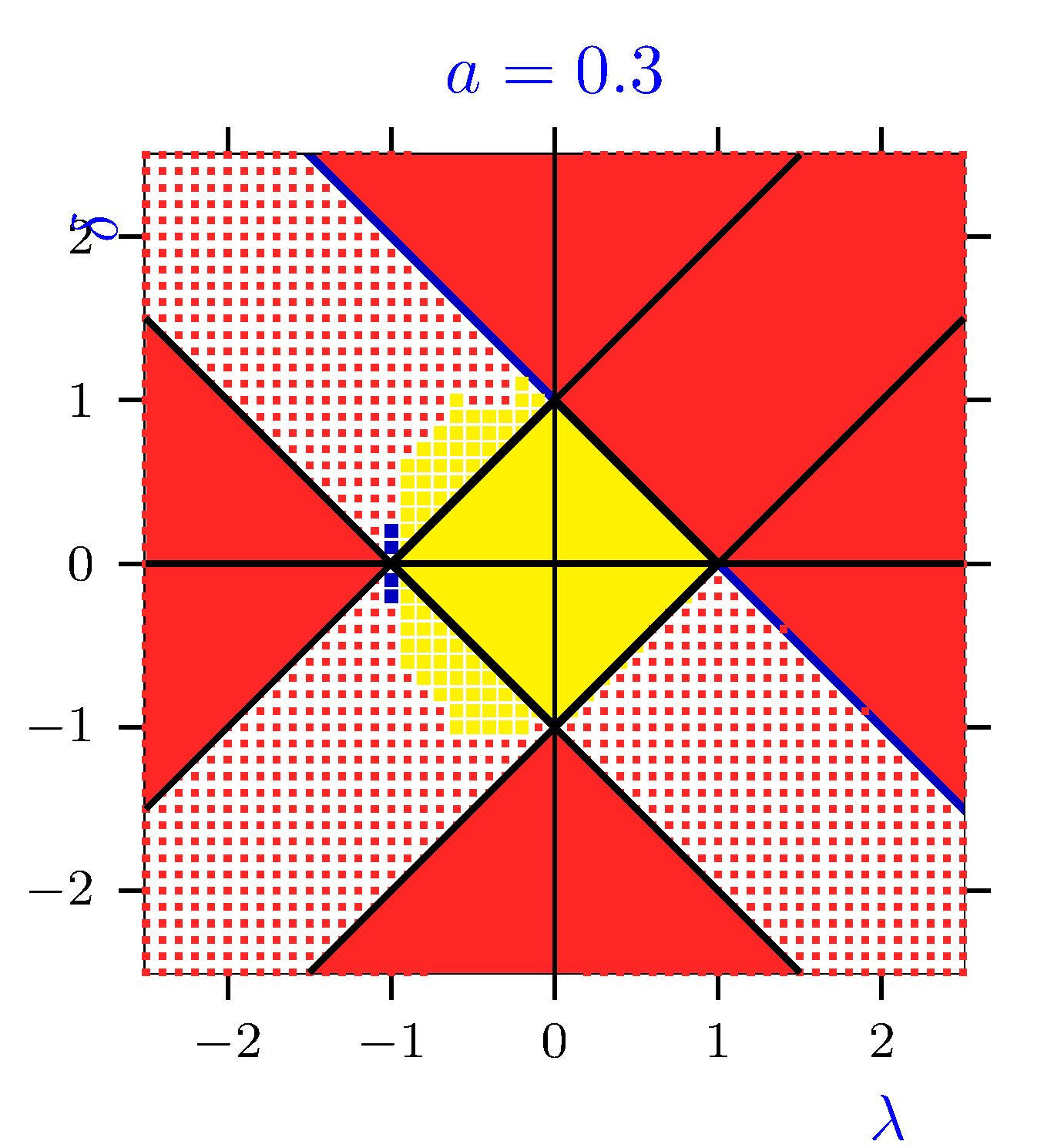}\hfill
    \includegraphics[width=0.48\textwidth]{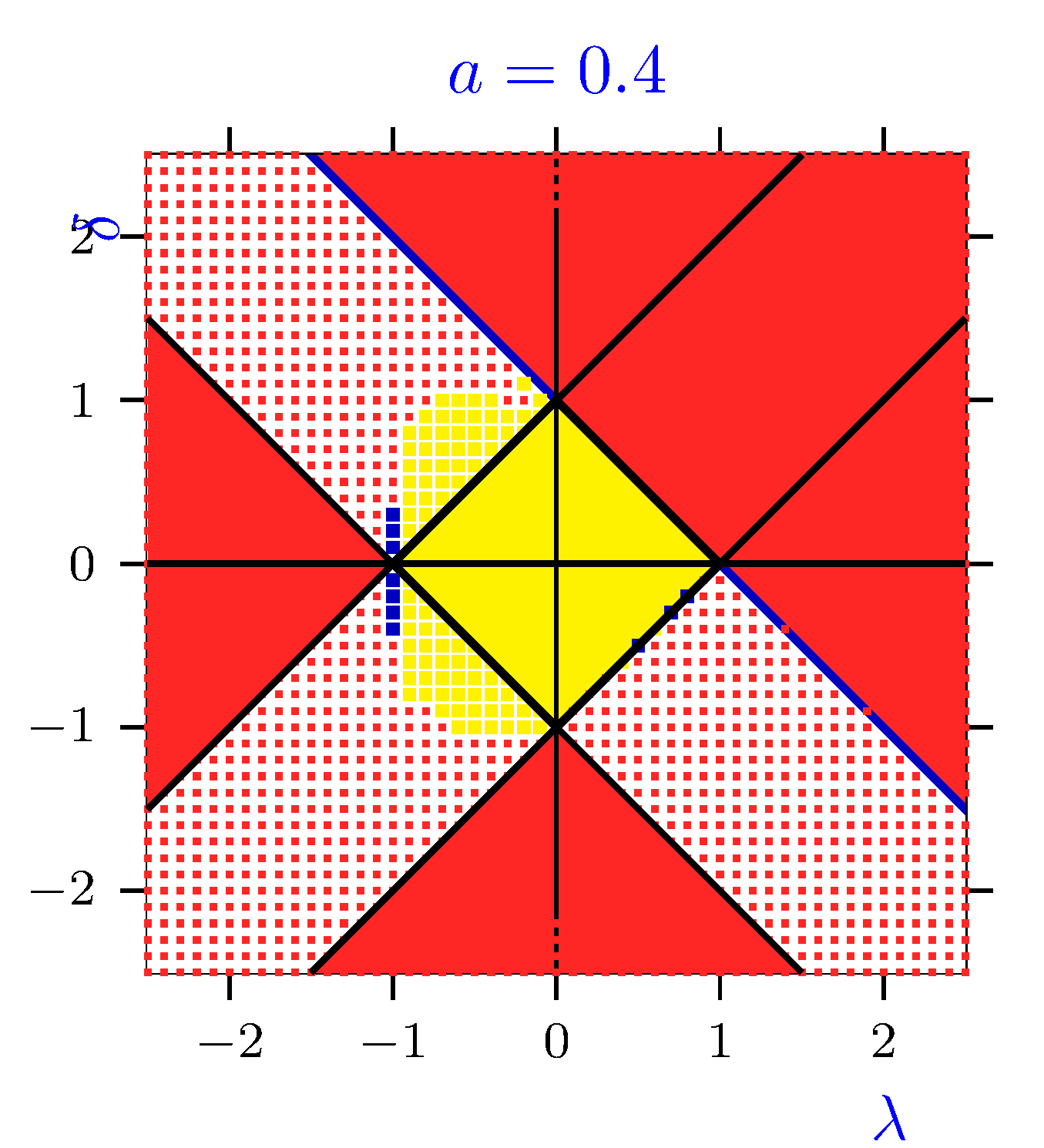}

    \medskip
    {\small
    \colorbox{yellow}{\strut\hspace{1em}} decaying to zero\quad
    \colorbox{red!85}{\strut\hspace{1em}} unstable/growing, including ultra-slow growth

    \smallskip
    \colorbox{blue!75!black}{\strut\hspace{1em}} exact constant solution or bounded candidate}
    \captionof{figure}{Refined analytical--numerical regime diagrams for (\ref{pmap}) with $h(n)=\lfloor an\rfloor$, $x_0=1$, and $a=0.1,0.2,\ldots,0.9$.}
    \label{fig:linear-regime-diagrams}
\end{center}

\begin{center}
    \includegraphics[width=0.48\textwidth]{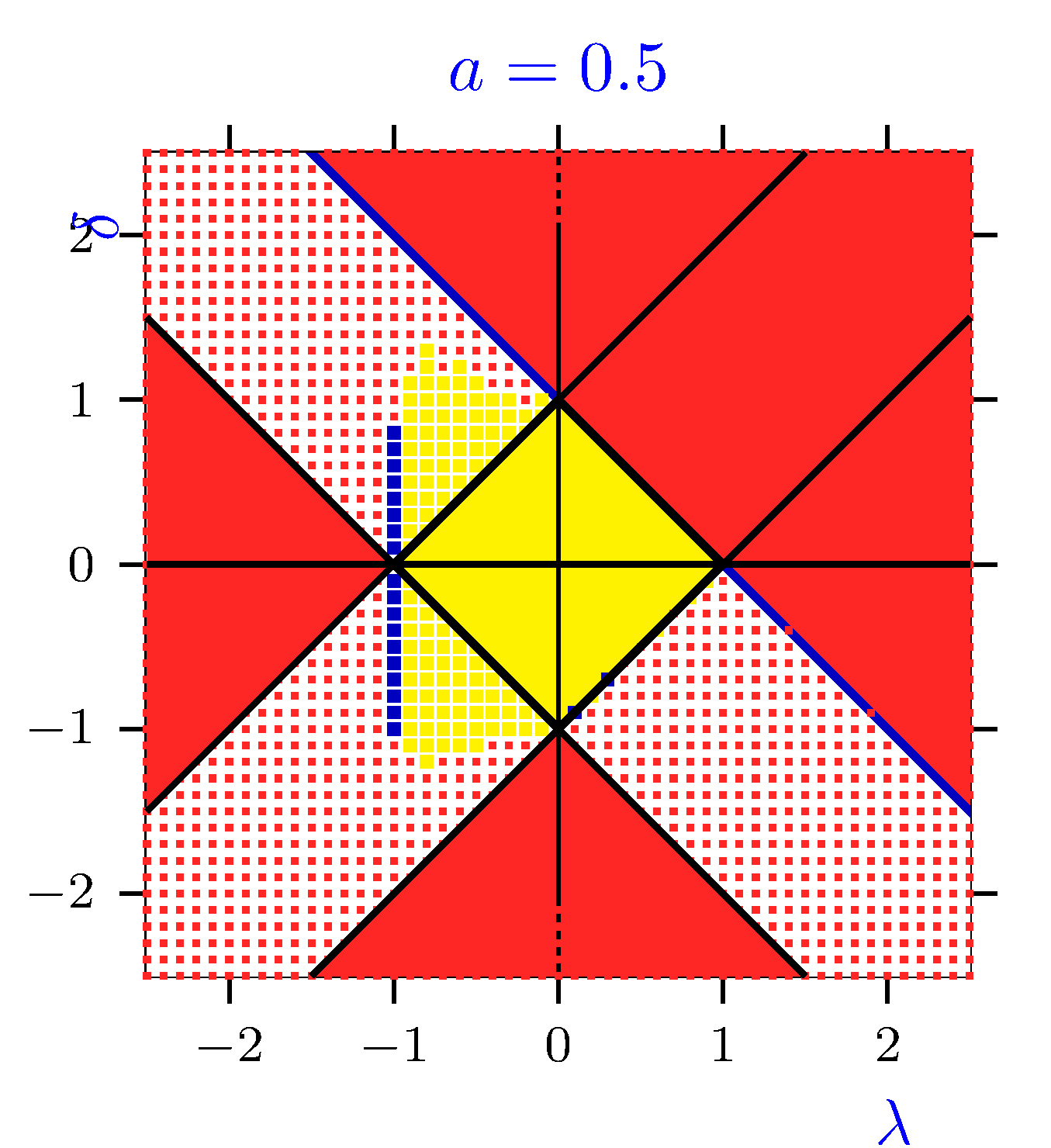}\hfill
    \includegraphics[width=0.48\textwidth]{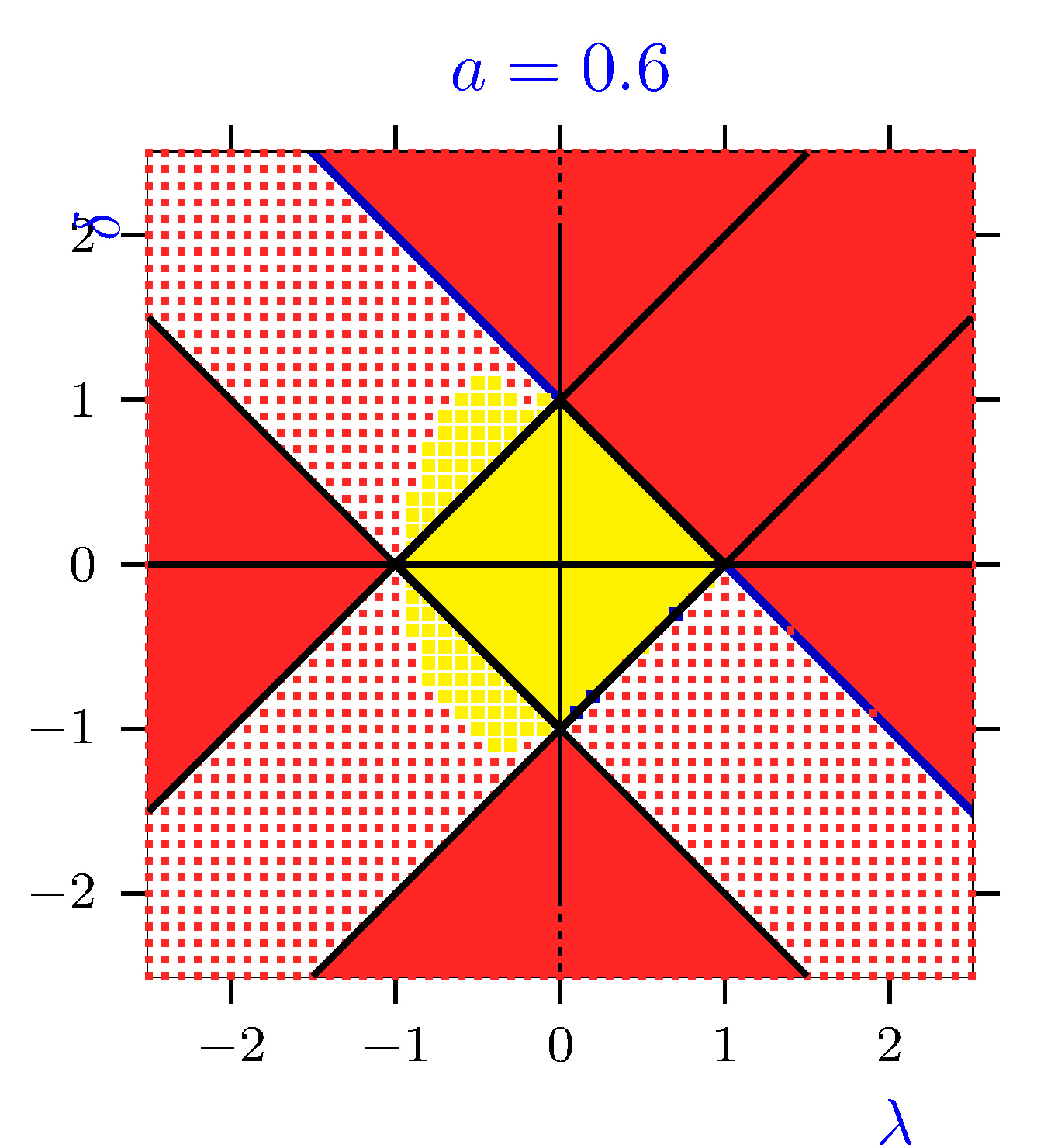}

    \medskip
    \includegraphics[width=0.48\textwidth]{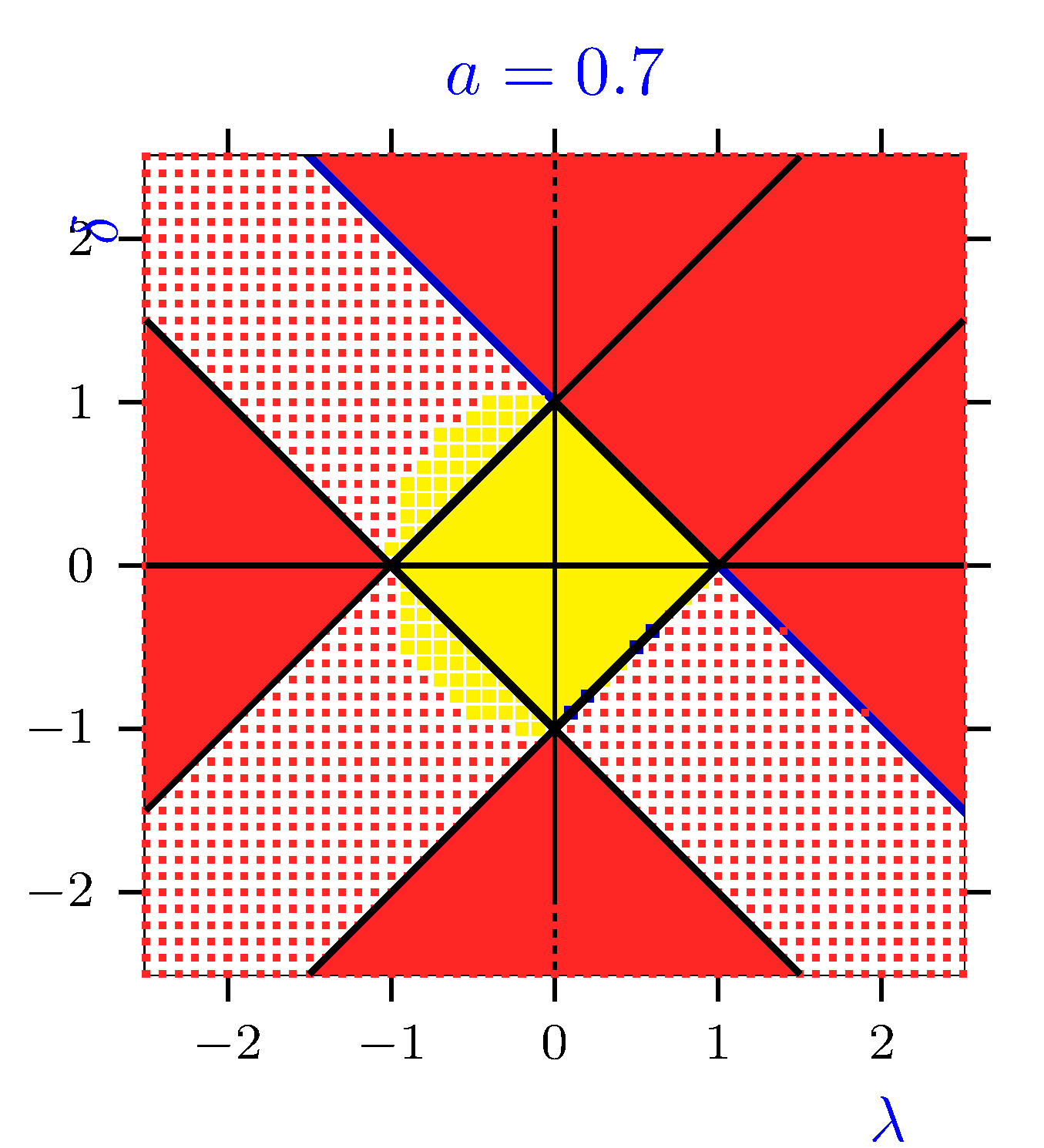}\hfill
    \includegraphics[width=0.48\textwidth]{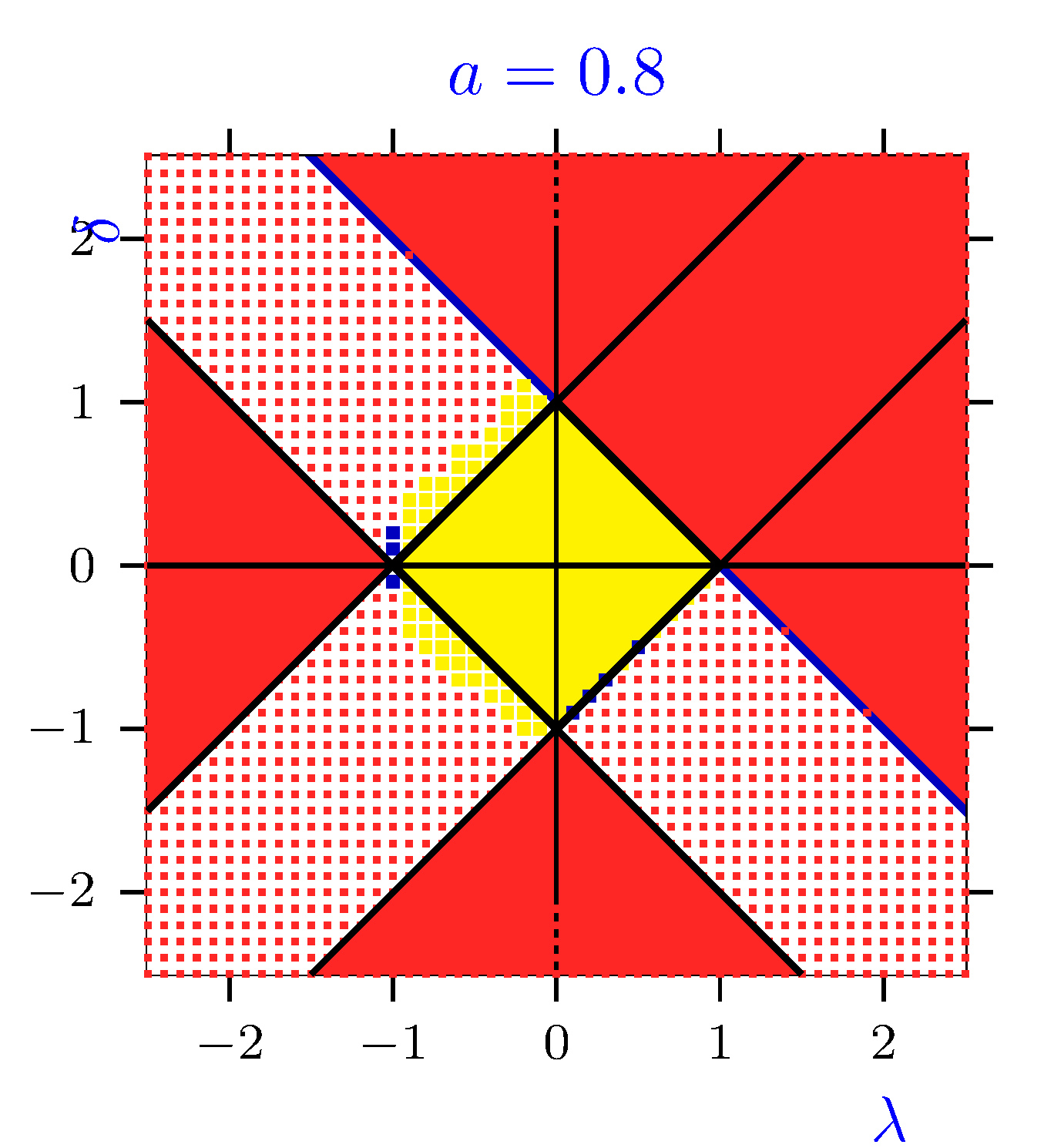}

    \small Figure \ref{fig:linear-regime-diagrams} continued: $a=0.5,0.6,0.7,0.8$, with every panel displayed on $[-2.5,2.5]^2$.
\end{center}

\begin{center}
    \includegraphics[width=0.42\textwidth]{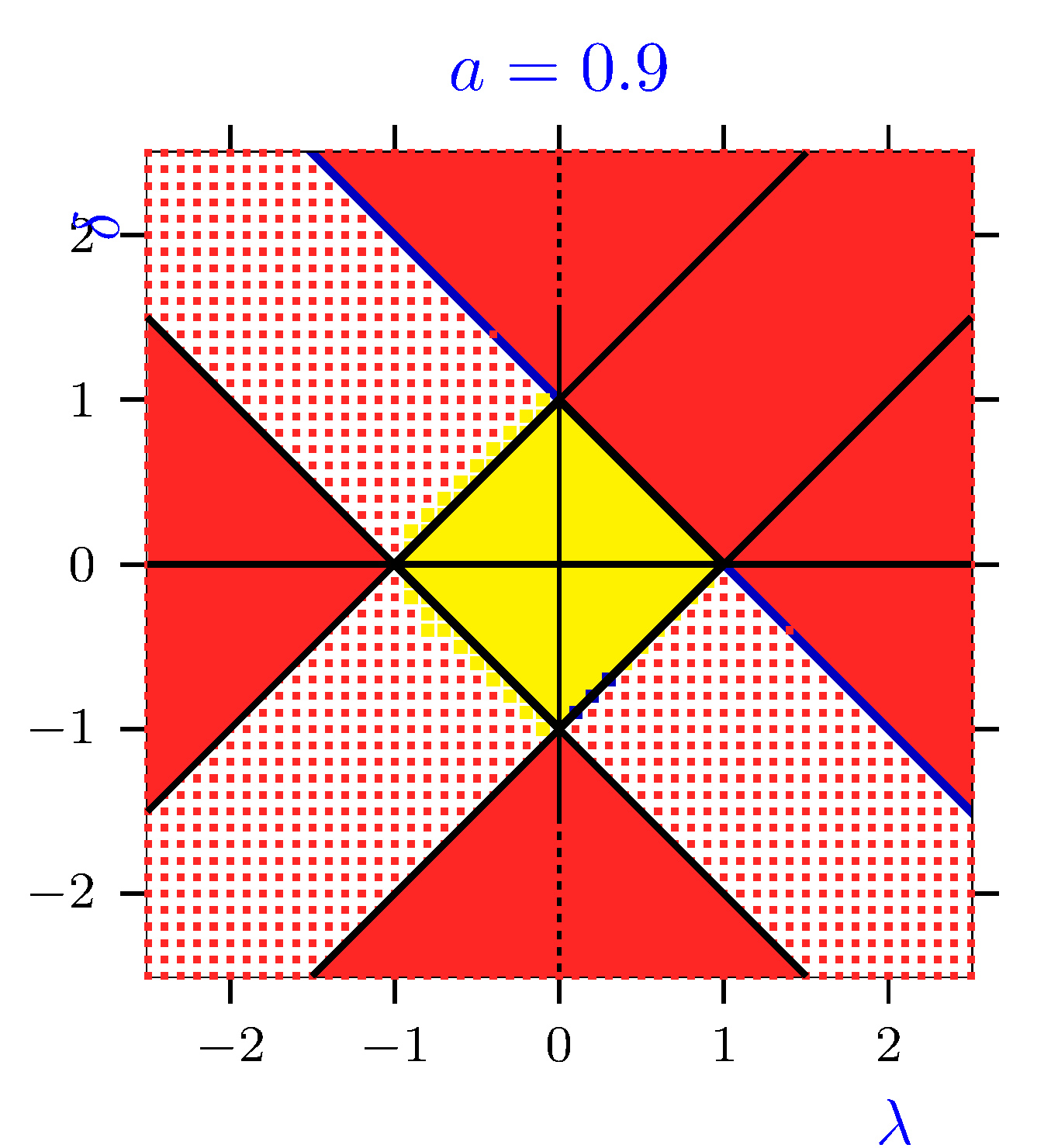}

    \small Figure \ref{fig:linear-regime-diagrams} continued: $a=0.9$, displayed on $[-2.5,2.5]^2$.
\end{center}

The solid diamond in each panel makes the role of Theorem \ref{thm1} explicit. Numerically decaying yellow points occur outside $|\lambda|+|\delta|<1$ for several values of $a$, showing that the theorem gives a sufficient, but not necessary, stability condition for proportional delays.
A comparison of the panels also reveals a clear numerical trend within the delay-dependent region: as $a$ increases from $0$ toward $1$, its stable portion shrinks while the unstable portion expands. Thus, within the parameter ranges and time horizons examined here, increasing the proportional-delay parameter $a$ tends to destabilize the system. 

Figure \ref{fig:concluding-stability-region} summarizes the resulting stability classification for $h(n)=\lfloor an\rfloor$, $0<a<1$, in the $(\lambda,\delta)$-plane. The yellow central region represents the delay-independent stable region, where solutions converge to zero. The red region represents parameter values for which instability is established independently of the particular value of $a$. The white regions are delay-dependent: stability or instability there cannot be decided from $(\lambda,\delta)$ alone and may change with $a$. Consequently, a white region should not be interpreted as stable or bounded; it identifies the part of the parameter plane requiring delay-specific analysis or numerical diagnostics such as those used in Figure \ref{fig:linear-regime-diagrams}.
For example, if we set $\lambda=-0.9$ and $\delta=0.3$ in the white region then the numerical experiments show that the trajectories grow when $a$ is close to 1 but may decay for smaller values of $a$. Figure \ref{fign04} shows numerical decay for $a=0.8$, whereas Figure \ref{fign05} shows growth for $a=0.9$.

\begin{center}
\begin{minipage}[t]{0.47\textwidth}
    \centering
    \includegraphics[width=0.96\linewidth]{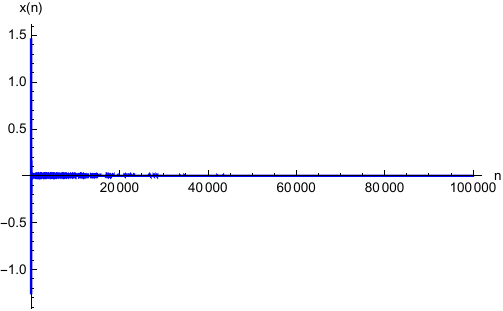}
    \captionof{figure}{Numerically decaying solution of (\ref{pmap}) with $h(n)=\lfloor 0.8 n\rfloor$, $\lambda=-0.9$, and $\delta=0.3$.}
    \label{fign04}
\end{minipage}
\hfill
\begin{minipage}[t]{0.47\textwidth}
    \centering
    \includegraphics[width=0.96\linewidth]{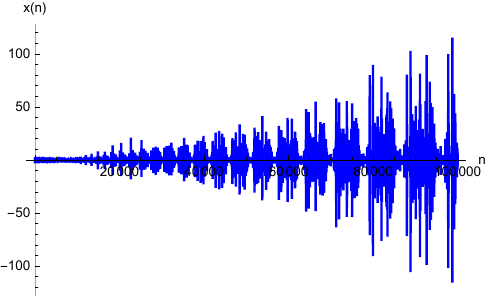}
    \captionof{figure}{Numerically growing solution of (\ref{pmap}) with $h(n)=\lfloor 0.9 n\rfloor$, $\lambda=-0.9$, and $\delta=0.3$.}
    \label{fign05}
\end{minipage}
\end{center}

\begin{figure}[!htbp]
    \centering
    \includegraphics[width=0.74\textwidth]{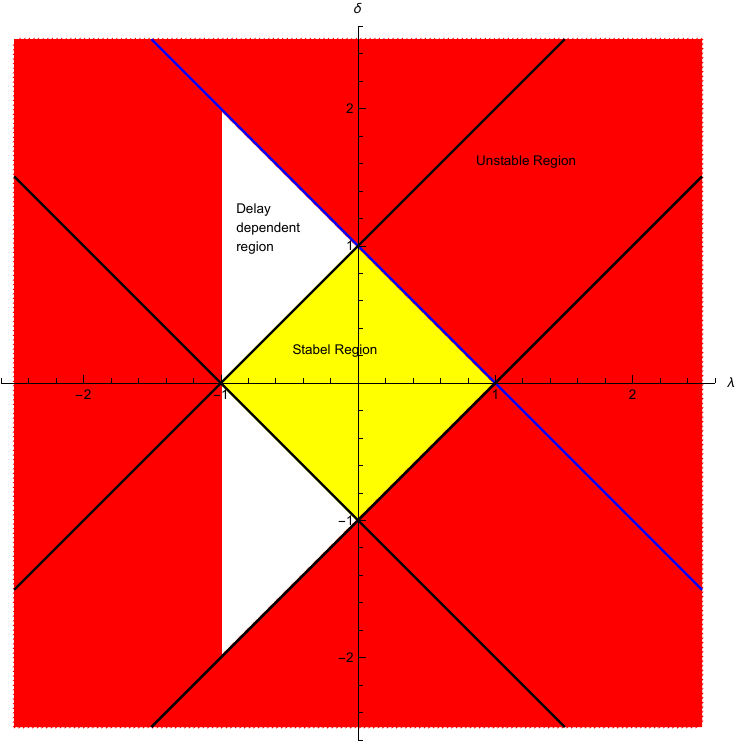}
    \captionof{figure}{Concluding stability diagram for (\ref{pmap}) with $h(n)=\lfloor an\rfloor$, $0<a<1$.}
    \label{fig:concluding-stability-region}
\end{figure}

\endgroup
\subsection{Nonlinear delayed maps}
\begingroup
We now consider the nonlinear delayed map
\begin{equation}
 x_{n+1} = (1-\epsilon)f(x_{n}) + \epsilon f(x_{h(n)}), \qquad n\geq 0. \label{nleq}
\end{equation}
Let $x^*$ be a fixed point of $f$, i.e., $x^*=f(x^*)$. Linearizing (\ref{nleq}) about $x^*$ gives
\begin{equation}
u_{n+1} = (1-\epsilon)f'(x^*)u_{n} + \epsilon f'(x^*)u_{h(n)}, \label{loclin}
\end{equation}
where $u_n=x_n-x^*$.

Comparing (\ref{loclin}) with equation (\ref{pmap}), we obtain $\lambda=(1-\epsilon)f'(x^*)$ and $\delta=\epsilon f'(x^*)$. Therefore, Theorem \ref{thm1} implies, for $h(n)\to\infty$, that the equilibrium point $x^*$ of (\ref{nleq}) is locally asymptotically stable if
\begin{equation}
|(1-\epsilon)f'(x^*)| + |\epsilon f'(x^*)| < 1. \label{stc}
\end{equation}
Since $0<\epsilon<1$, condition (\ref{stc}) simplifies to
\begin{equation}
|f'(x^*)| < 1. \label{sstc}
\end{equation}
Thus, (\ref{sstc}) guarantees local asymptotic stability of $x^*$ independently of the particular diverging delay function and of the coupling parameter $\epsilon$.

\subsection{Example: delayed logistic map}
Consider (\ref{nleq}) with $f(x)=\mu x(1-x)$, the logistic map with delay. The equilibrium points are $x_1^*=0$ and $x_2^*=(\mu-1)/\mu$. Moreover, $f'(x_1^*)=\mu$ and $f'(x_2^*)=2-\mu$.

\subsubsection*{Sufficient conditions for stability}
Using (\ref{sstc}), we obtain the following delay-independent sufficient conditions for diverging delays.
\begin{itemize}
\item If $-1 < \mu < 1$, then $x_1^* = 0$ is locally asymptotically stable (LAS).
\item If $1 < \mu < 3$, then $x_2^* = \frac{\mu-1}{\mu}$ is LAS.
\end{itemize}

\subsection*{Remarks}

\begin{itemize}
\item These are sufficient conditions only; stability may depend on the delay $h$ and coupling strength $\epsilon$ when they are not satisfied.
\item In the stable region, convergence is faster when $\epsilon$ is close to $0$.
\item For $h(n)=\lfloor an\rfloor$, $0<a<1$, convergence in the stable region is faster when $a$ is close to 1.
\end{itemize}

\endgroup

\section{Model}\label{sec3}
The analytical results of the previous section establish general stability properties for pantograph difference equations. We now turn to specific nonlinear systems to investigate how proportional delays influence global dynamics beyond local stability analysis. In particular, we introduce pantograph versions of the logistic and tent maps and study their transient behavior, multistability, and chaotic evolution.

Consider the evolution rule $x_{i+1}=(1-\epsilon)f(x_i) +\epsilon f(x_{h(i)})$, as in (\ref{nleq}). We consider two delay functions, $h(i)=\lfloor ai \rfloor$ and $h(i)=\lfloor ai^{\alpha}\rfloor$, with
$0<a<1$ and $0<\alpha<1$. This assignment ensures that $h(i)\le i$. We consider
$0\le \epsilon \le 1$. If an interval $I$ satisfies $f(I)\subseteq I$, then the evolution stays in $I$ whenever $x_0\in I$, because each update is a convex combination of two values in $I$. In particular, if $f:[0,1] \rightarrow [0,1]$, the evolution stays in the unit interval if $x_0\in [0,1]$. For $\epsilon=1$, it is a homogeneous pantograph difference equation. For $\epsilon=0$, we
recover the original difference equation. For $0<\epsilon<1$, it is an inhomogeneous pantograph 
difference equation.

For $\epsilon=1$, the delayed trajectory follows the same successive iterates of $f$ as the undelayed map, but repeats each value over a long interval. Here $f^m$ denotes the $m$-fold composition of $f$. For $h(i)=\lfloor i/4\rfloor$, the stated recurrence gives
\[
\begin{aligned}
x_1=\cdots=x_4&=f(x_0),\\
x_5=\cdots=x_{20}&=f^2(x_0),\\
x_{21}=\cdots=x_{84}&=f^3(x_0).
\end{aligned}
\]
More generally, define $T_m=(4^m-1)/3$ for $m\geq0$. Then
\[
x_n=f^m(x_0),\qquad T_m\leq n<T_{m+1},
\]
and the interval corresponding to the $m$th iterate has length $T_{m+1}-T_m=4^m$. These endpoints follow from $T_{m+1}=4T_m+1$; the shift by one is essential because the evolution rule specifies $x_{i+1}$, not $x_i$. Consequently, an asymptotically stable fixed point of the undelayed map remains asymptotically stable in this homogeneous case, although convergence may be unobservable within a practical computation time.

For $f(x)=bx$, the exact solution is
\[
x_n=b^{\lfloor\log_4(3n+1)\rfloor}x_0.
\]
For $b>1$ and $x_0\ne0$, its magnitude has power-law order $n^{\log_4 b}$, with a bounded step-like modulation. Thus, exponential growth in the undelayed map becomes power-law growth in physical iteration time. 
For \(b>1\) and \(x_0\ne0\), the magnitude of the solution has
power-law scaling \(n^{\log_4 b}\), with a bounded step-like
modulation. Thus, exponential growth in the undelayed map becomes
power-law growth in physical iteration time.  The corresponding smooth guiding curves for \(x_0=1\)
are \(\sqrt{3t}\) and \(3t\), respectively. These curves capture the
power-law exponents and prefactors of the smooth approximation obtained
by replacing the floor function by its argument; they are not exact
solutions. The factor \(3\) arises from the exact solution
$
x_n=b^{\lfloor\log_4(3n+1)\rfloor} \sim (3n+1)^{\log_4 b},
$
as smooth approximation.
 If we ignore 1 for large $n$, we get $\sqrt{3t}$ and $3t$ respectively for $b=2$ and $b=4$ (see Figure~\ref{lser}).
For the more general delay $h(i)=\lfloor i/k\rfloor$, where $k\geq2$ is an integer, the corresponding endpoints are $T_m=(k^m-1)/(k-1)$, and the interval lengths are $k^m$.

The logistic map is given by $f(x)=\mu x(1-x)$ with $\mu\in [0,4]$. We assign
$\mu=4$ and choose $h(i)=\lfloor i/4\rfloor$. 
This exponential increase in periods
where the variable remains unchanged till it moves to the next value is
shown in Fig.~\ref{tserl}. More generally, if an undelayed orbit reaches a prescribed accuracy after 100 iterations in a convergent parameter regime, the homogeneous delayed orbit reaches the same iterate only at $T_{100}=(4^{100}-1)/3$. This convergence example does not describe the generic logistic orbit at $\mu=4$. The homogeneous delayed orbit and the corresponding undelayed orbit have the same set of accumulation points, but their time scales and time averages need not agree.

\begin{figure}
	\includegraphics[width=\linewidth]{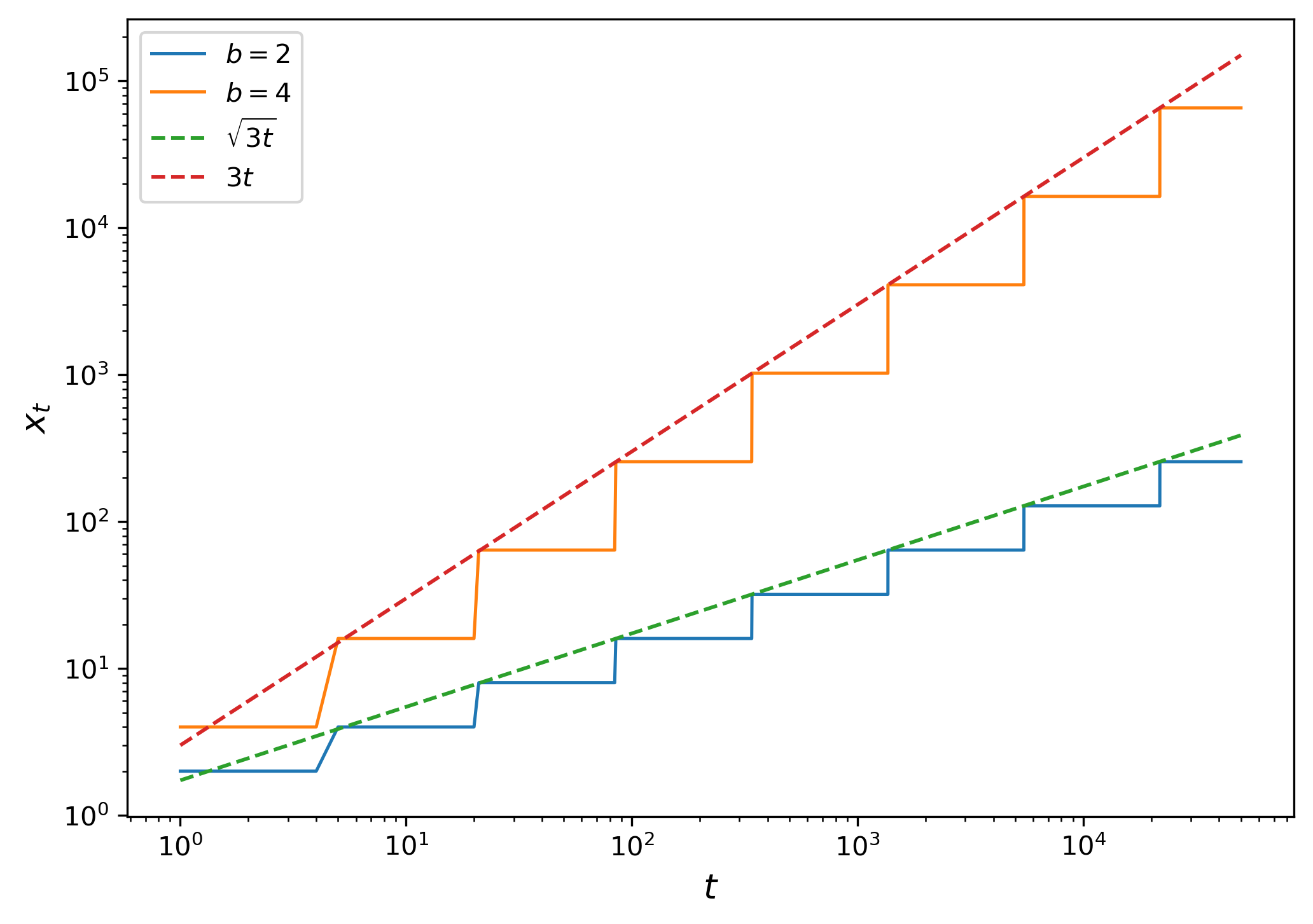}
	\caption{Time series of the homogeneous linear pantograph map with
\(x_0=1\), \(h(i)=\lfloor i/4\rfloor\), and \(\epsilon=1\), plotted on
logarithmic scales on both axes for \(f(x)=bx\), with \(b=2\) and \(b=4\).
The guiding curves \(\sqrt{3t}\) and \(3t\) illustrate the power-law
exponents \(\log_4 2=1/2\) and \(\log_4 4=1\), respectively; they are
not exact solutions. The periods of stasis grow exponentially with the
composition index.}
	\label{lser}
\end{figure}

\begin{figure}
	\includegraphics[width=\linewidth]{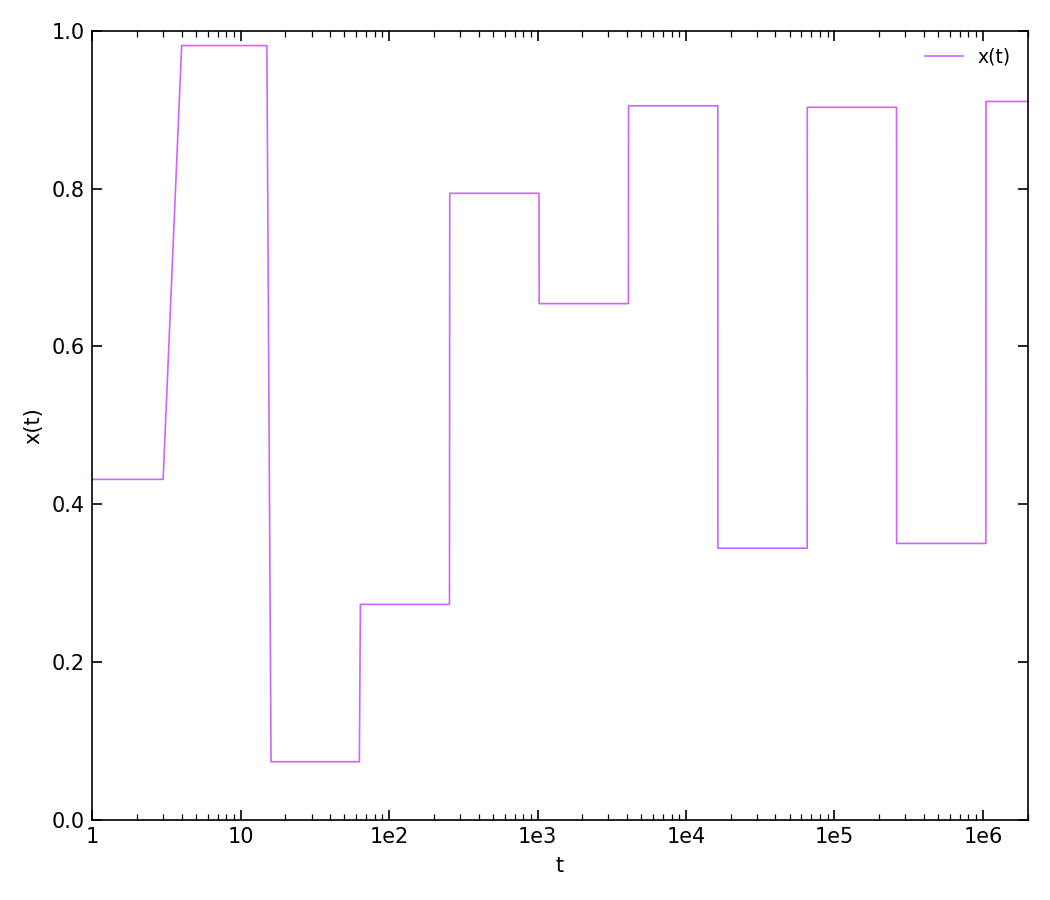}
	\caption{Evolution of the homogeneous pantograph logistic map with $\mu=4$, $x_0=0.123$, $h(i)=\lfloor i/4\rfloor$, and $\epsilon=1$. Long periods of stasis become asymptotically equally spaced on the logarithmic time axis.}
	\label{tserl}
\end{figure}

\begin{figure}
	\includegraphics[width=\linewidth]{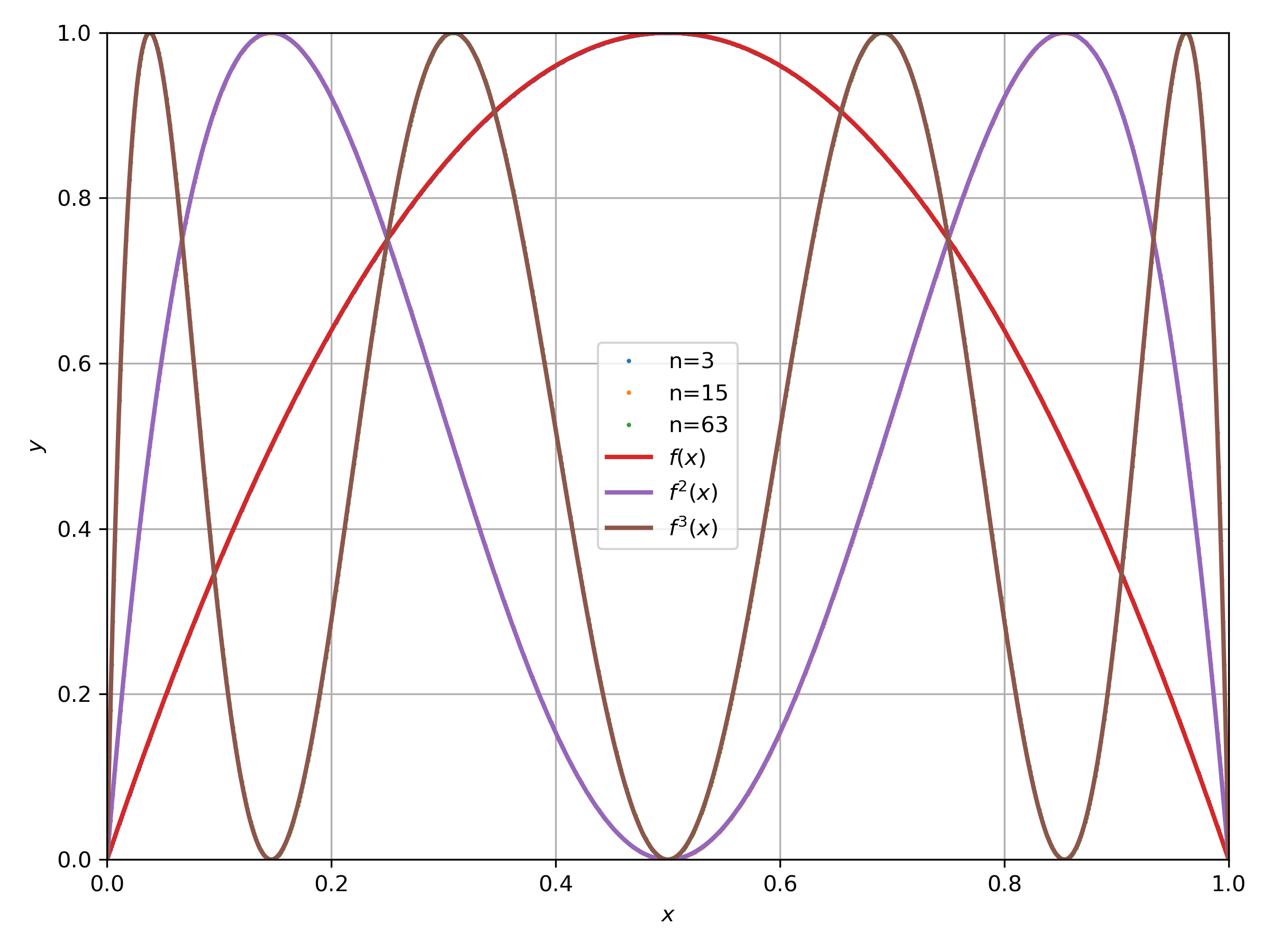}
	\caption{Dependence of $x_3$, $x_{15}$, and $x_{63}$ on $x_0$ for the homogeneous logistic map with $\mu=4$, $h(i)=\lfloor i/4\rfloor$, and $\epsilon=1$. These values agree with $f(x_0)$, $f^2(x_0)$, and $f^3(x_0)$, respectively.}
	\label{match}
\end{figure}

If we plot $x(3)$ as a function of $x(0)$, it is simply $f(x(0))$. Similarly,
$x(15)=f^2(x(0))$ and $x(63)=f^3(x(0))$. This is what is numerically observed 
and is shown in Fig. \ref{match}.

For a sublinear delay, the stasis periods grow even faster, as illustrated in Fig.~\ref{tserloghom}. In particular, for $h(i)=\lfloor\sqrt{i}/2\rfloor$ and $\epsilon=1$, the plateau endpoints satisfy
\[
T_1=1,\qquad T_{m+1}=4T_m^2+1,
\]
so that $T_2=5$, $T_3=101$, $T_4=40805$, and $T_5=6660192101$. As before, $x_n=f^m(x_0)$ for $T_m\leq n<T_{m+1}$. The plateau lengths grow faster than exponentially in $m$. In particular, $x_{10^8}=f^4(x_0)$, as illustrated for the logistic map in Fig.~\ref{neplog}. The delay coefficient matters: for $h(i)=\lfloor\sqrt{i}/4\rfloor$, the endpoint recurrence is instead $T_{m+1}=16T_m^2+1$, giving $T_2=17$, $T_3=4625$, and $T_4=342250001$, so that $x_{10^8}=f^3(x_0)$. In either case, the homogeneous orbit retains the accumulation points of the undelayed orbit, with a much slower progression through its iterates.

\begin{figure}
	\includegraphics[width=\linewidth]{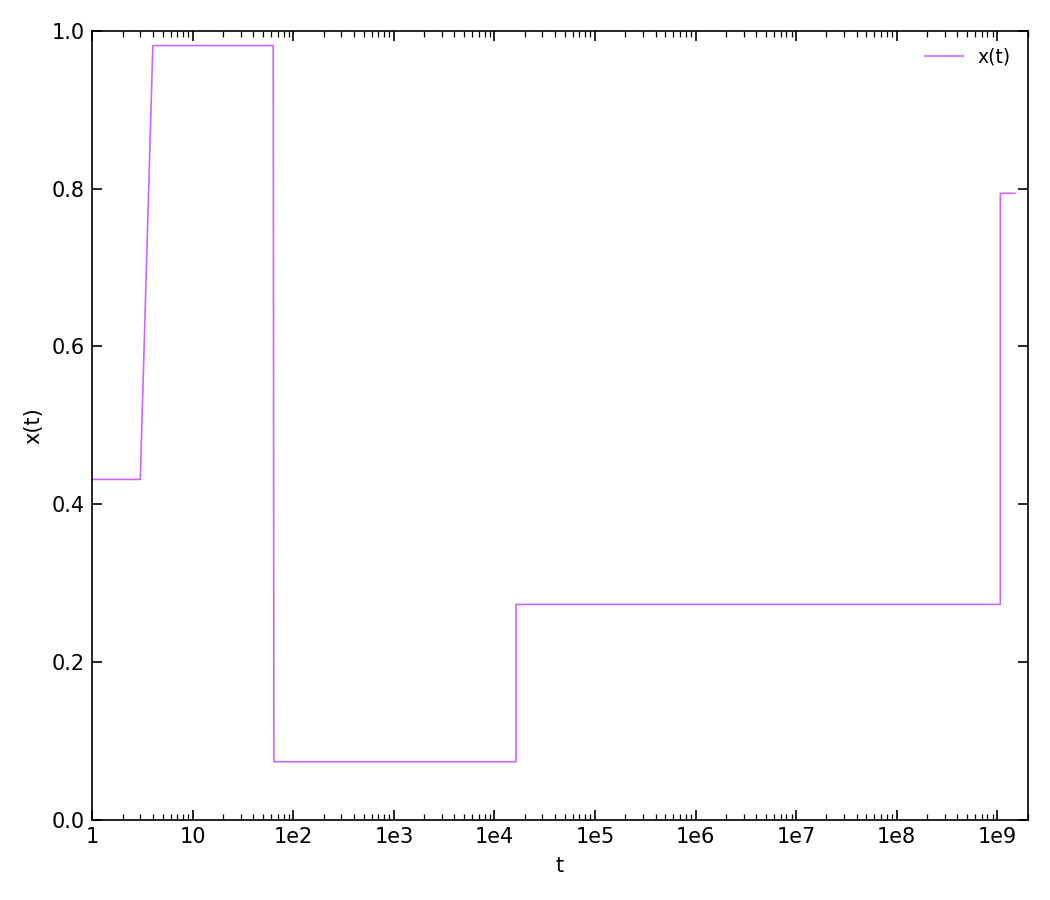}
	\caption{Evolution of the homogeneous pantograph logistic map with $\mu=4$, $x_0=0.123$, and $\epsilon=1$ under a sublinear proportional delay. The long periods of stasis are not equally spaced even on the logarithmic time axis.}
	\label{tserloghom}
\end{figure}

\begin{figure}
	\includegraphics[width=\linewidth]{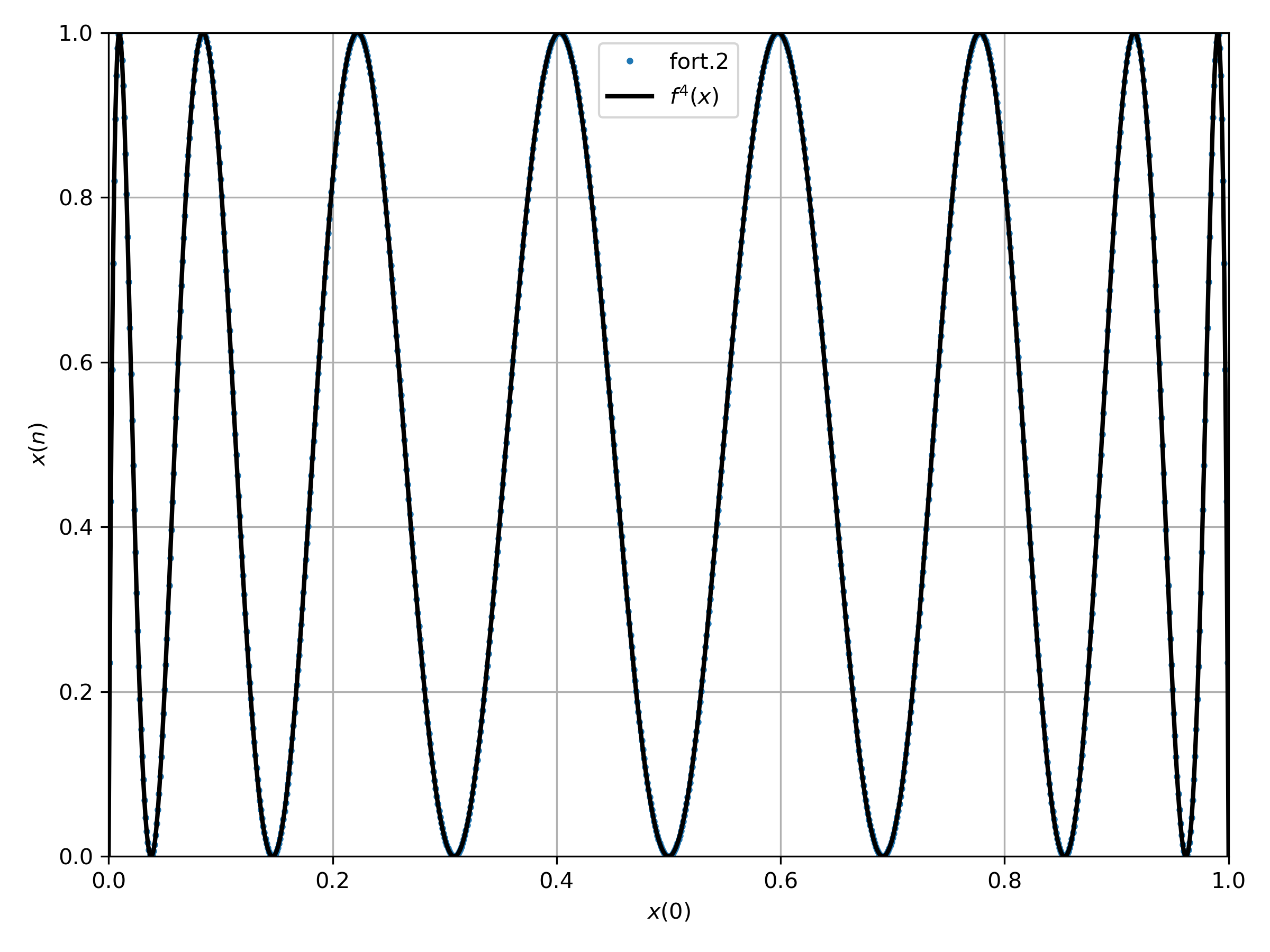}
	\caption{Finite-time dependence on the initial condition for the homogeneous pantograph logistic map with $\mu=4$, $h(i)=\lfloor\sqrt{i}/2\rfloor$, and $\epsilon=1$. The sampled state after $10^8$ iterations is plotted against $x_0$ and agrees with $f^4(x_0)$.}
	\label{neplog}
\end{figure}

We next consider $0<\epsilon<1$ for the logistic map with $\mu=4$ and $h(i)=\lfloor i/4\rfloor$. Figure~\ref{bif-log} shows finite-time bifurcation diagrams as functions of $\epsilon$ for $x_0=0.845$, $0.056$, and $0.347$, using 100 successive states after $10^5$ iterations. The diagrams differ near the homogeneous limit $\epsilon=1$. At smaller values of $\epsilon$, the sampled ranges largely overlap and show no apparent finite-time multistability for these initial conditions.

Figures~\ref{ep.3} and~\ref{ep.9} show the dependence on $x_0$ at $\epsilon=0.3$ and $0.9$, respectively, using the same stated sampling window. At $\epsilon=0.3$, the sampled range is approximately independent of the initial condition, whereas at $\epsilon=0.9$ it depends strongly on $x_0$. The symmetry about $x_0=0.5$ follows from $f(x)=f(1-x)$: complementary initial conditions produce identical delayed trajectories from the first update onward. The observed dependence near $\epsilon=1$ is evidence of finite-time multistability, not proof of distinct asymptotic attractors. As in the homogeneous case, long transients may account for this behavior, although the exact composition formula no longer applies when $0<\epsilon<1$.

\begin{figure}
	\includegraphics[width=\linewidth]{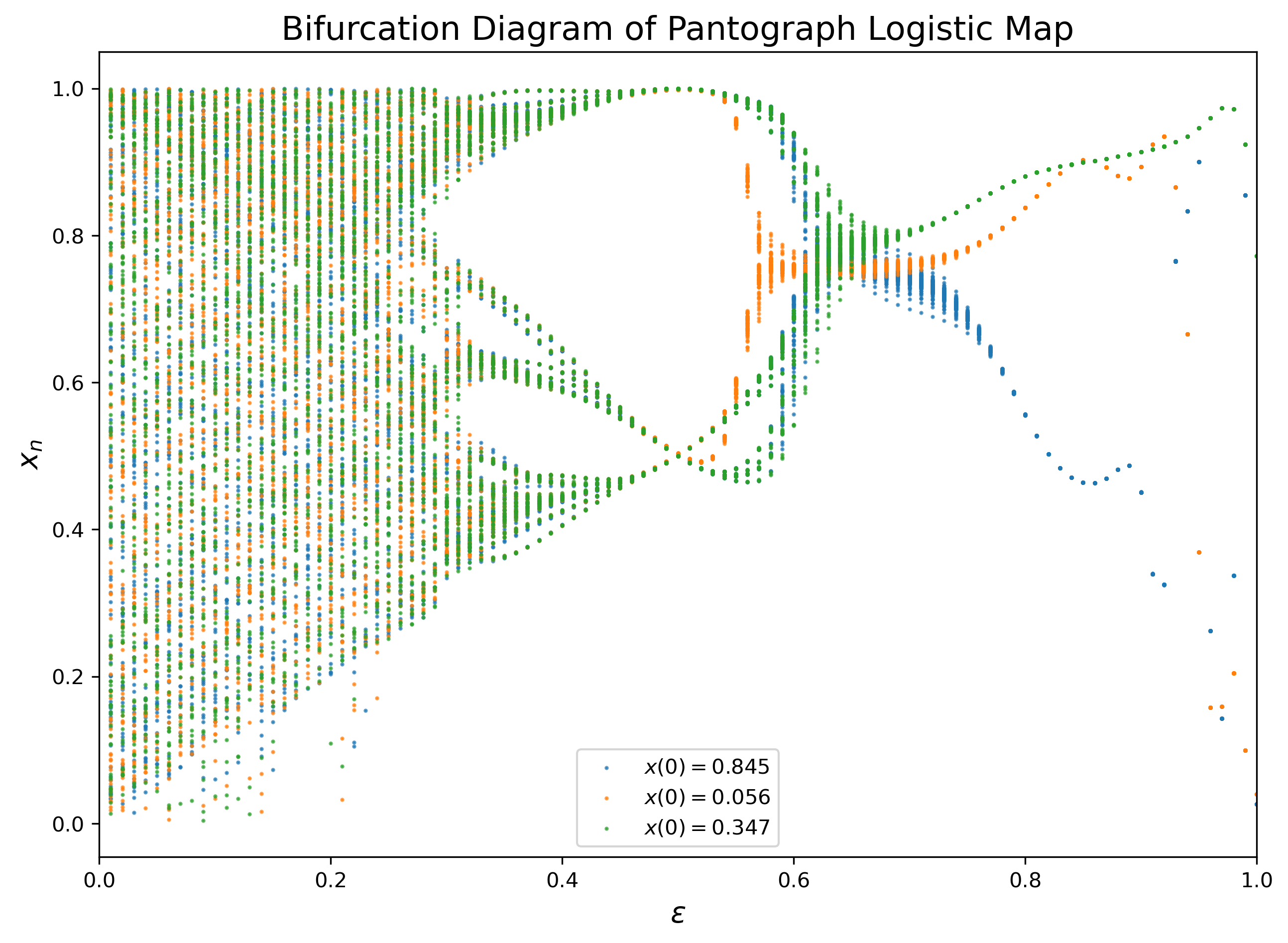}
	\caption{Finite-time bifurcation diagram of the pantograph logistic map with $\mu=4$ and $h(i)=\lfloor i/4\rfloor$ as a function of $\epsilon$. After $10^5$ iterations, 100 successive states are plotted for $x_0=0.845$, $0.056$, and $0.347$.}
	\label{bif-log}
\end{figure}

\begin{figure}
	\includegraphics[width=\linewidth]{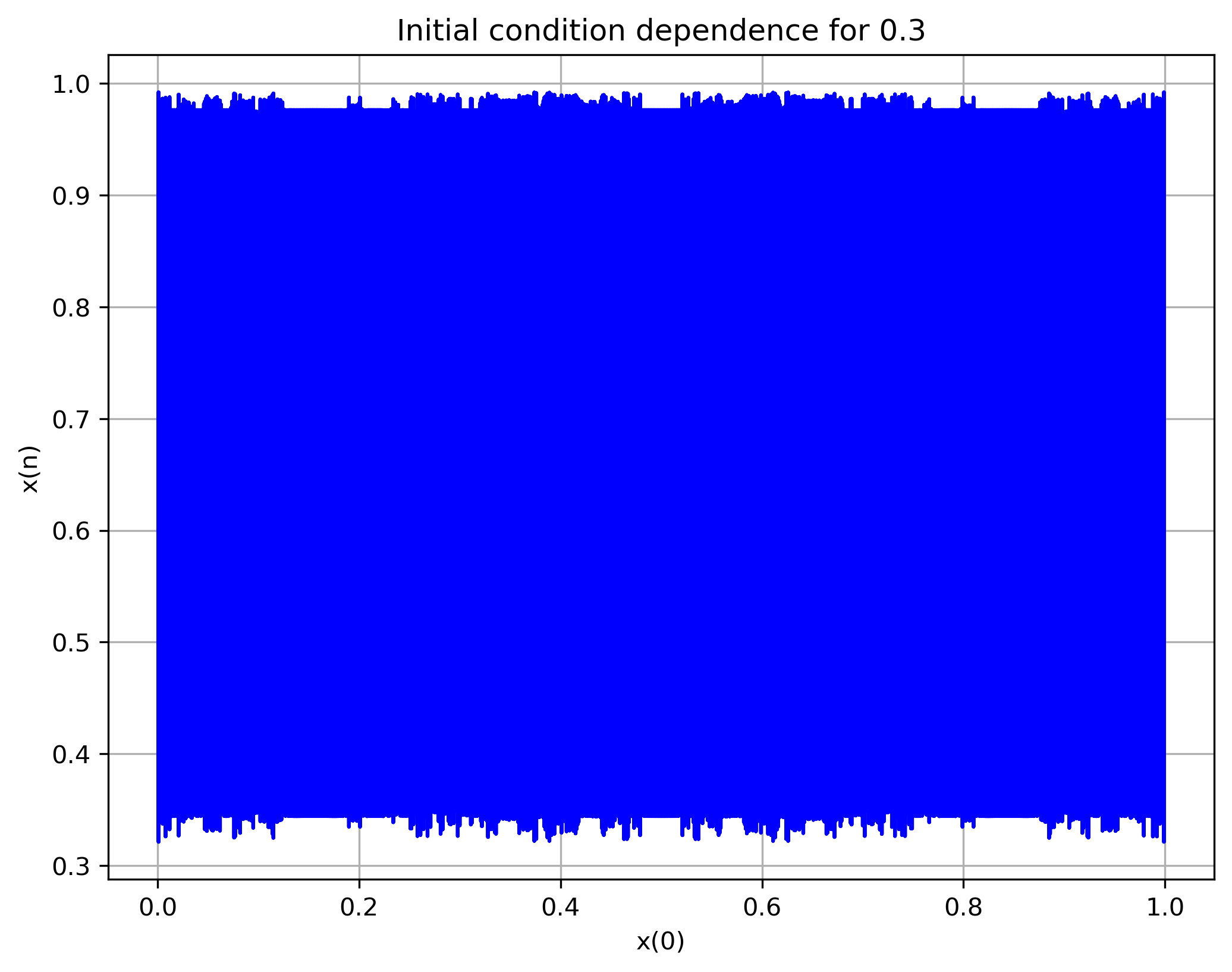}
	\caption{Finite-time dependence on $x_0$ for the pantograph logistic map with $\mu=4$, $h(i)=\lfloor i/4\rfloor$, and $\epsilon=0.3$. After $10^5$ iterations, 100 successive states are plotted against $x_0$.}
	\label{ep.3}
\end{figure}
\begin{figure}
	\includegraphics[width=\linewidth]{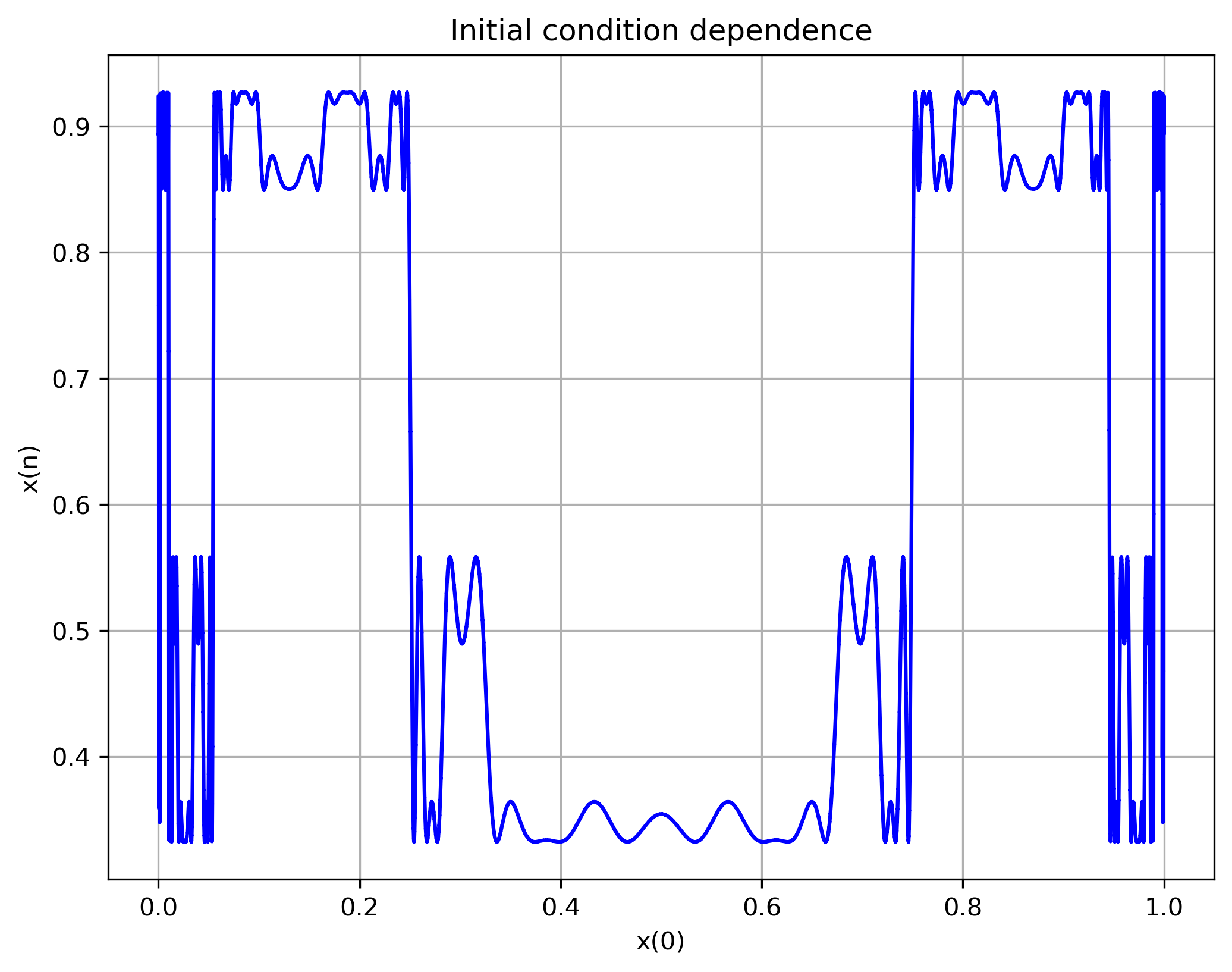}
	\caption{Finite-time dependence on $x_0$ for the pantograph logistic map with $\mu=4$, $h(i)=\lfloor i/4\rfloor$, and $\epsilon=0.9$. After $10^5$ iterations, 100 successive states are plotted against $x_0$.}
	\label{ep.9}
\end{figure}

For the sublinear delay $h(i)=\lfloor\sqrt{i}/2\rfloor$, Fig.~\ref{nlin} shows the logistic-map bifurcation diagrams for $x_0=0.088$, $0.38$, and $0.95$. The sampled states again depend strongly on the initial condition near $\epsilon=1$. At $\epsilon=1$, Fig.~\ref{nep1} shows the dependence on $x_0$ after $10^5$ iterations. Since this observation window lies between $T_4=40805$ and $T_5=6660192101$, the plotted values equal $f^4(x_0)$.

\begin{figure}
	\includegraphics[width=\linewidth]{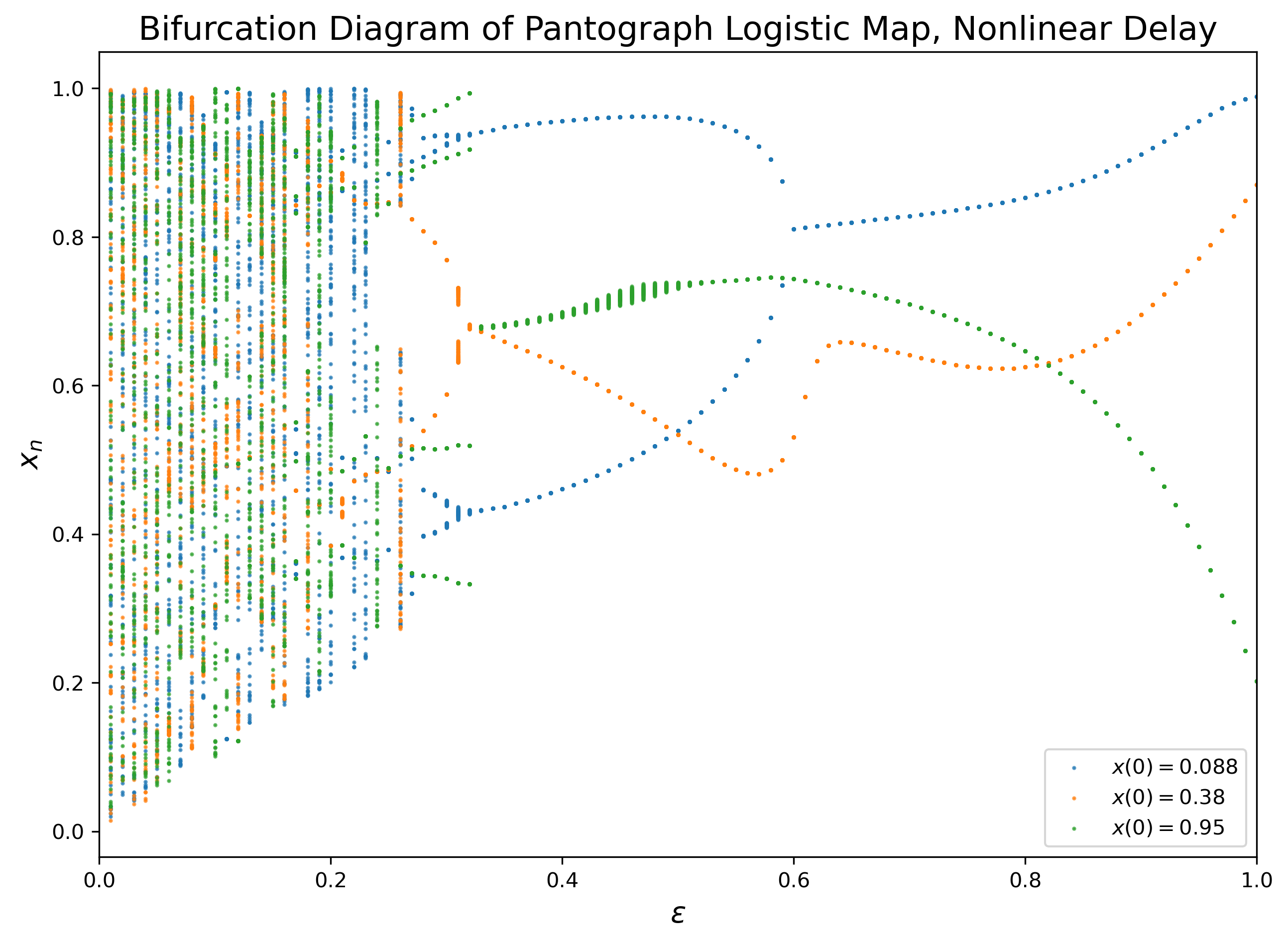}
	\caption{Finite-time bifurcation diagram of the pantograph logistic map with $\mu=4$ and $h(i)=\lfloor\sqrt{i}/2\rfloor$ as a function of $\epsilon$, for $x_0=0.088$, $0.38$, and $0.95$. After $10^5$ iterations, 100 successive states are plotted for each value of $\epsilon$.}
	\label{nlin}
\end{figure}
\begin{figure}
	\includegraphics[width=\linewidth]{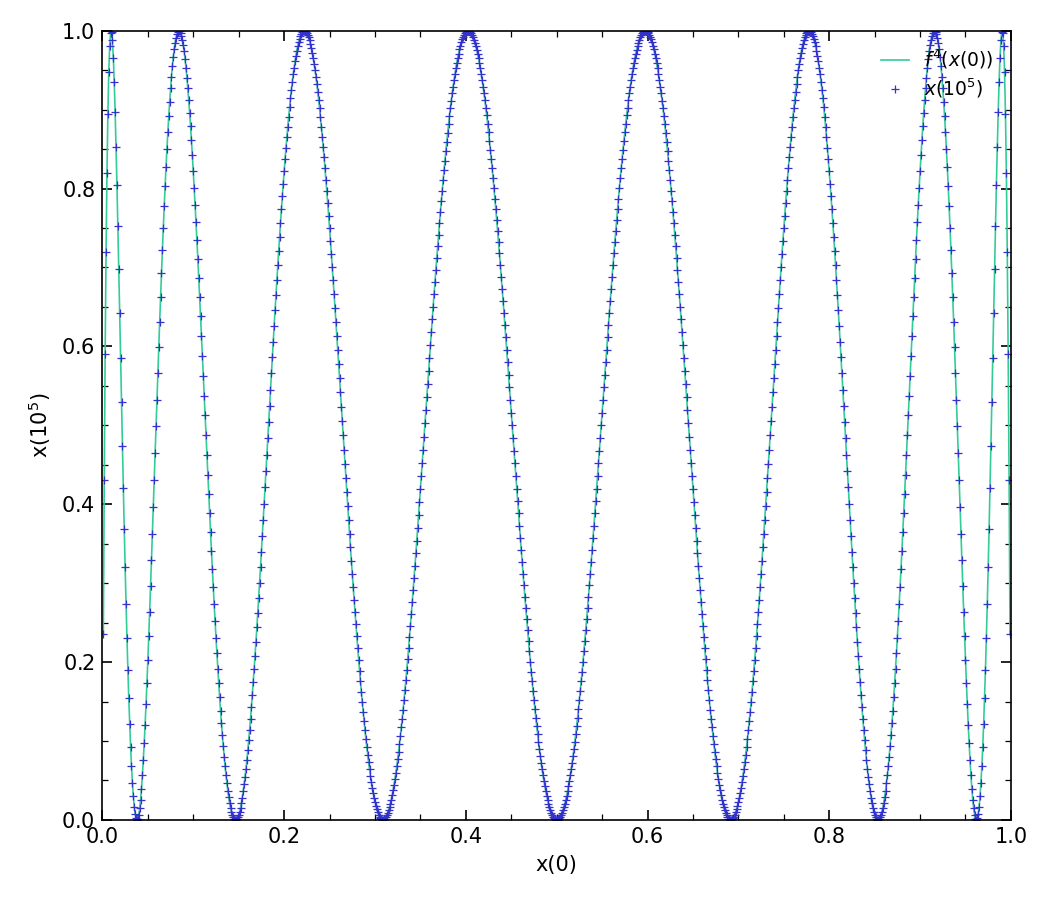}
	\caption{Finite-time dependence on $x_0$ for the homogeneous pantograph logistic map with $\mu=4$, $h(i)=\lfloor\sqrt{i}/2\rfloor$, and $\epsilon=1$. The 100 successive states sampled after $10^5$ iterations all equal $f^4(x_0)$.}
	\label{nep1}
\end{figure}

We now consider the pantograph version of the symmetric tent map,
\[
f(x)=\begin{cases}
cx, & 0\leq x\leq 1/2,\\
c(1-x), & 1/2<x\leq1,
\end{cases}
\qquad c=2,
\]
which maps $[0,1]$ onto itself. Figure~\ref{tenth} shows the bifurcation diagrams as functions of $\epsilon$ for $x_0=0.196$, $0.49$, and $0.94$, with $h(i)=\lfloor i/4\rfloor$ in the upper panel and $h(i)=\lfloor\sqrt{i}/2\rfloor$ in the lower panel. The sampled trajectories show initial-condition dependence for larger coupling strengths, consistent with finite-time multistability; the plots alone do not establish an exact threshold or asymptotic coexistence of attractors.

For $\epsilon=1$ and $h(i)=\lfloor i/4\rfloor$, Fig.~\ref{bif-tent} plots $x_{50}$ and $x_{100}$ against $x_0$. These finite-time initial-condition maps equal $f^3(x_0)$ and $f^4(x_0)$, respectively. At $n=10^5$, however, the state is $f^9(x_0)$, since $T_9=87381\leq10^5<T_{10}=349525$.

\begin{figure}
	\centering
	\includegraphics[width=0.9\linewidth]{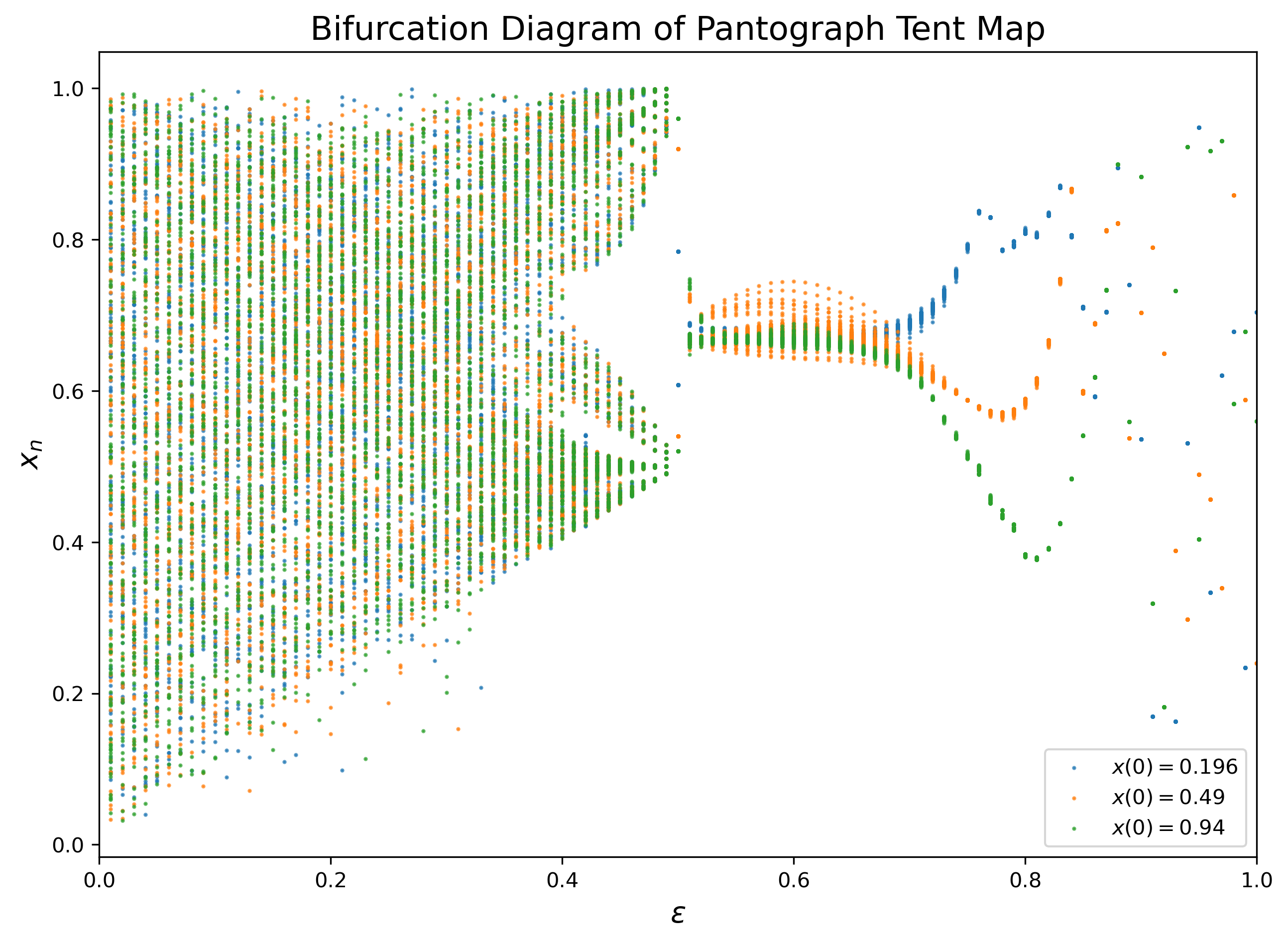}
	\includegraphics[width=0.9\linewidth]{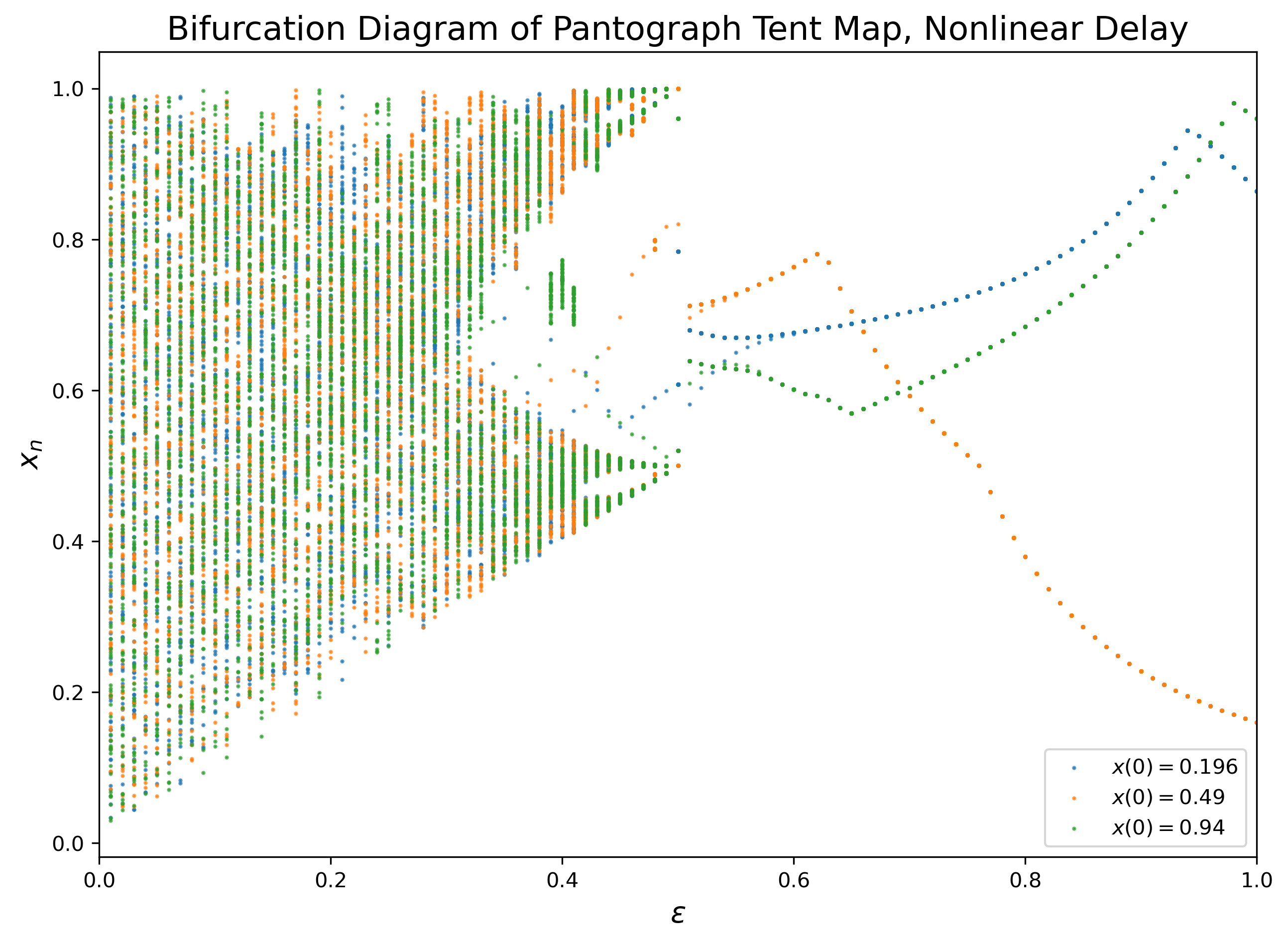}
	\caption{Finite-time bifurcation diagrams of the pantograph tent map with $c=2$ as functions of $\epsilon$, for $x_0=0.196$, $0.49$, and $0.94$. Upper panel: proportional delay $h(i)=\lfloor i/4\rfloor$. Lower panel: sublinear delay $h(i)=\lfloor\sqrt{i}/2\rfloor$.}
	\label{tenth}
\end{figure}

\begin{figure}
	\includegraphics[width=\linewidth]{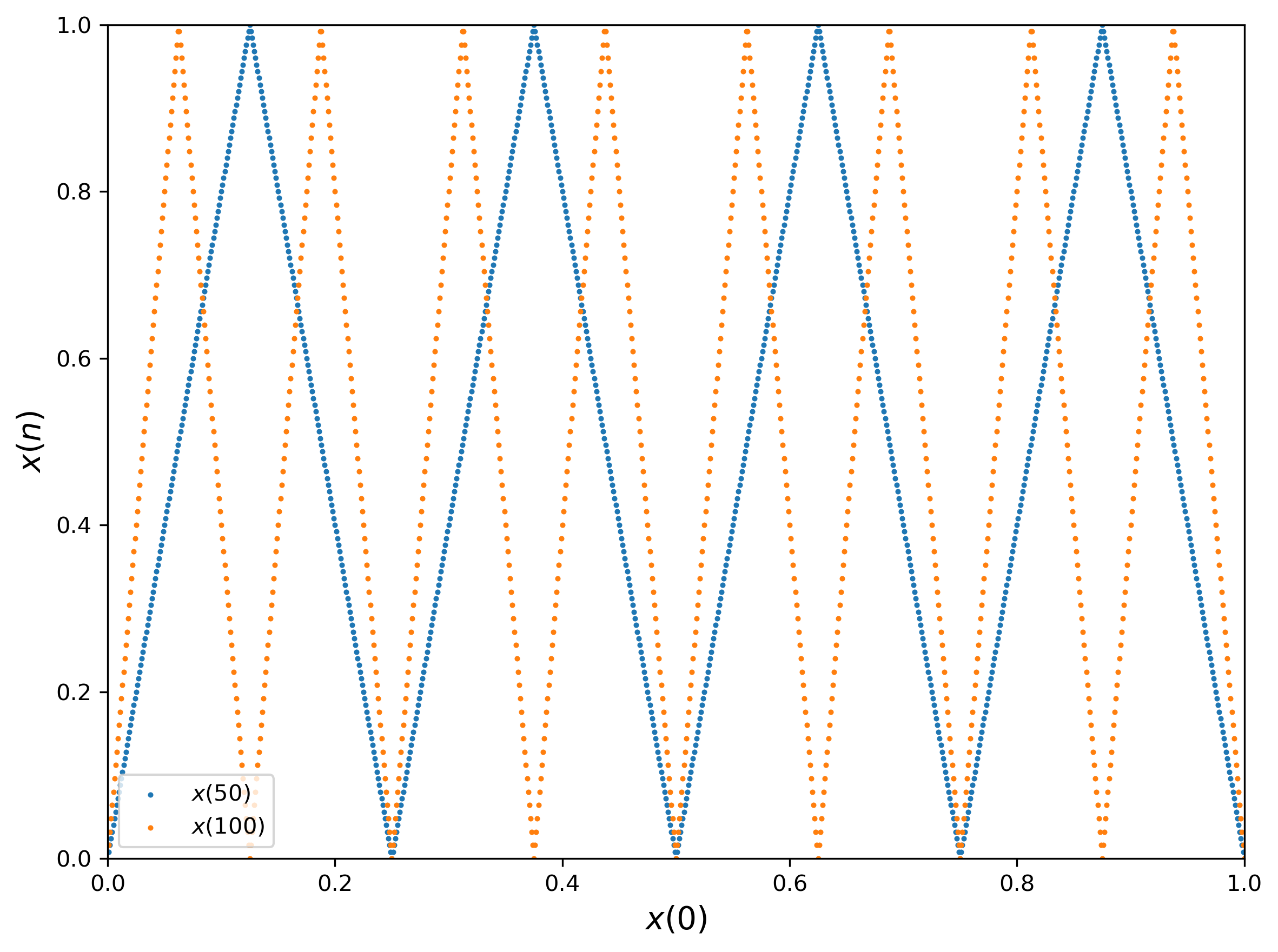}
	\caption{Finite-time initial-condition maps $x_0\mapsto x_n$ for the homogeneous pantograph tent map with $c=2$, $h(i)=\lfloor i/4\rfloor$, and $\epsilon=1$, at $n=50$ and $n=100$. These are $f^3(x_0)$ and $f^4(x_0)$, respectively, rather than consecutive-state return maps.}
	\label{bif-tent}
\end{figure}

For the sublinear delay $h(i)=\lfloor\sqrt{i}/2\rfloor$ and $\epsilon=1$, Fig.~\ref{ntent} plots $x_{10^3}=f^3(x_0)$ and $x_{10^7}=f^4(x_0)$ against $x_0$. Figure~\ref{nhtent} shows that the state sampled after $10^8$ iterations still equals $f^4(x_0)$, in agreement with the plateau endpoints derived above.
\begin{figure}
	\includegraphics[width=\linewidth]{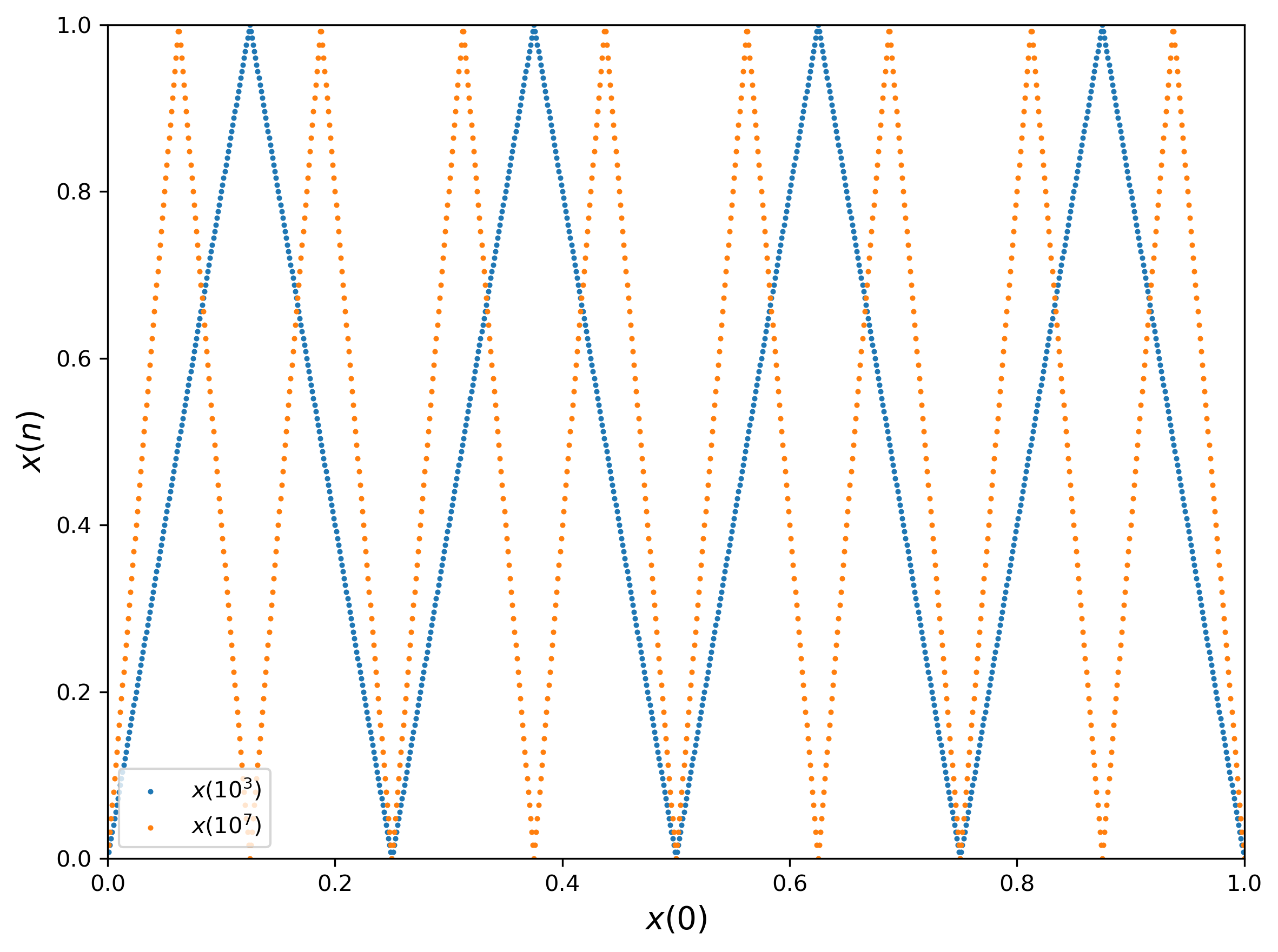}
	\caption{Finite-time initial-condition maps $x_0\mapsto x_n$ for the homogeneous pantograph tent map with $c=2$, $h(i)=\lfloor\sqrt{i}/2\rfloor$, and $\epsilon=1$, at $n=10^3$ and $n=10^7$. These equal $f^3(x_0)$ and $f^4(x_0)$, respectively.}
	\label{ntent}
\end{figure}

\begin{figure}
	\includegraphics[width=\linewidth]{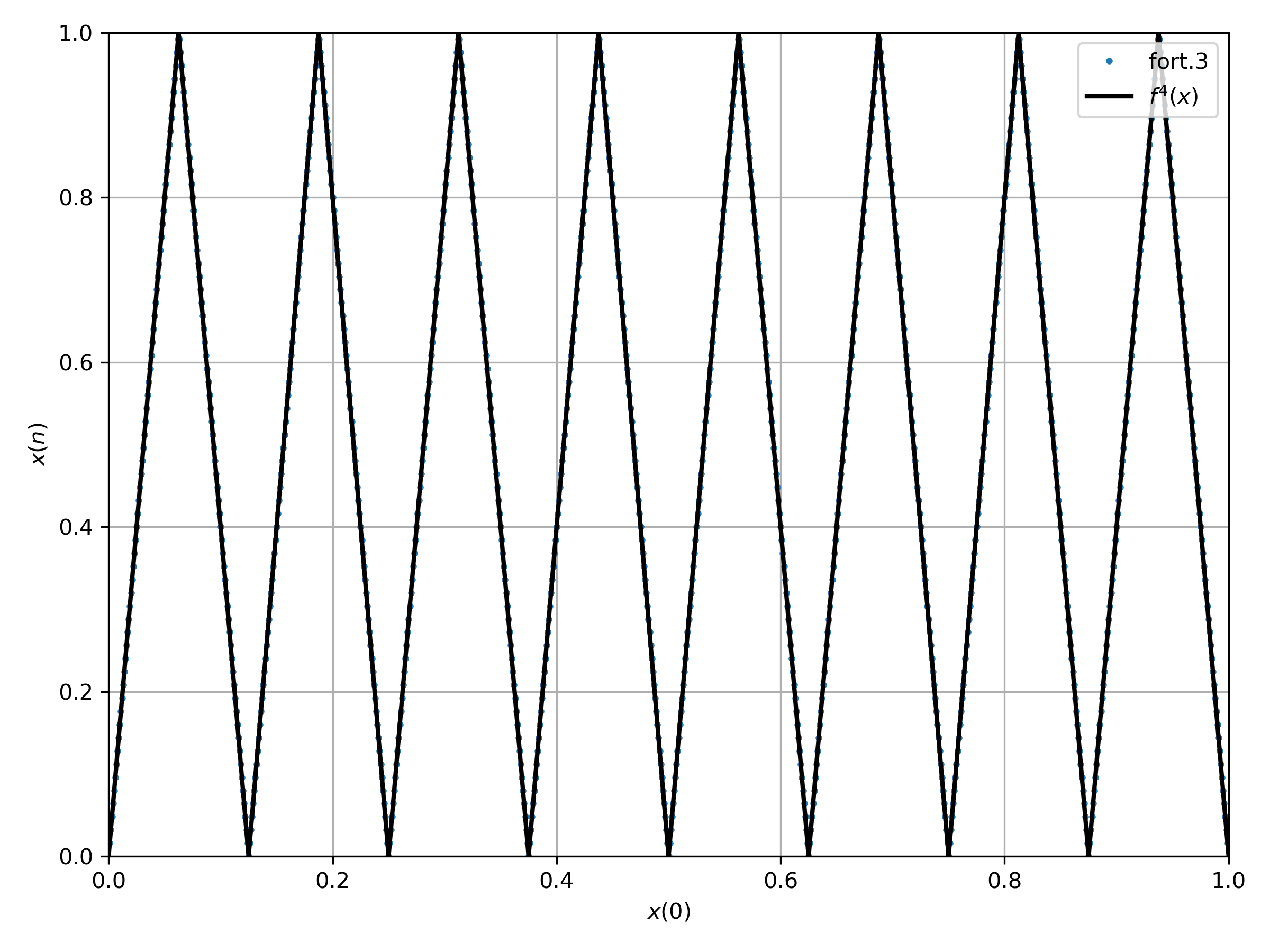}
	\caption{Finite-time dependence on $x_0$ for the homogeneous pantograph tent map with $c=2$, $h(i)=\lfloor\sqrt{i}/2\rfloor$, and $\epsilon=1$. The state sampled after $10^8$ iterations is compared with $f^4(x_0)$.}
	\label{nhtent}
\end{figure}

For smaller values of $\epsilon$, the diagrams show no apparent finite-time multistability for the sampled initial conditions. To investigate sensitivity to perturbations, the system is first evolved for $N=10^5$ steps, after which a nearby trajectory $y$ is evolved using the same delayed rule. The stated perturbations are uniformly distributed in $[-0.001,0.001]$. We denote the mean separation after $j$ further updates by
\[
d(j)=\left\langle |x_{N+j}-y_{N+j}|\right\rangle,
\]
where the average is taken over $10^3$ configurations. Figure~\ref{lyap} shows this separation for $h(i)=\lfloor i/4\rfloor$ and $\epsilon=0.1$ over 25 further updates. The initial growth is approximately exponential, followed by saturation.

For the logistic map with the sublinear delay $h(i)=\lfloor\sqrt{i}/4\rfloor$ and $\epsilon=0.05$, Fig.~\ref{lyap2} shows similar behavior over 20 further updates. These experiments provide evidence of finite-time exponential separation, consistent with chaotic dynamics, but do not by themselves establish a positive asymptotic Lyapunov exponent or a chaotic attractor. Such a conclusion requires a specified perturbation protocol for the delayed state and a long-time analysis of infinitesimal perturbation growth.

\begin{figure}
	\includegraphics[width=\linewidth]{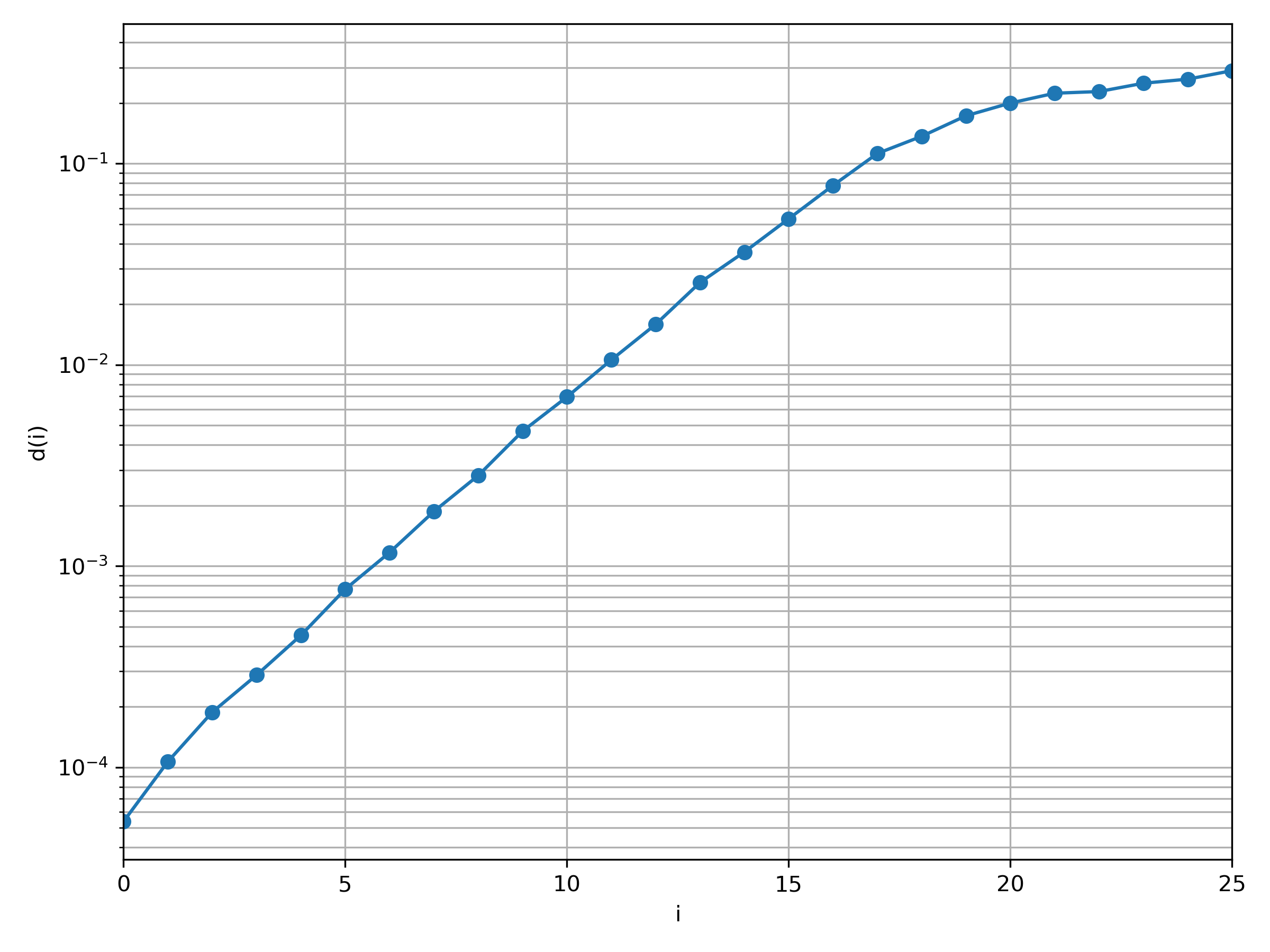}
	\caption{Mean separation $d(j)$ of reference and perturbed trajectories after an initial evolution of $10^5$ steps, for $h(i)=\lfloor i/4\rfloor$ and $\epsilon=0.1$. The plot covers 25 further updates and averages over $10^3$ configurations. On the logarithmic vertical scale, the initial approximately linear segment indicates finite-time exponential separation before saturation.}
	\label{lyap}
\end{figure}

\begin{figure}
	\includegraphics[width=\linewidth]{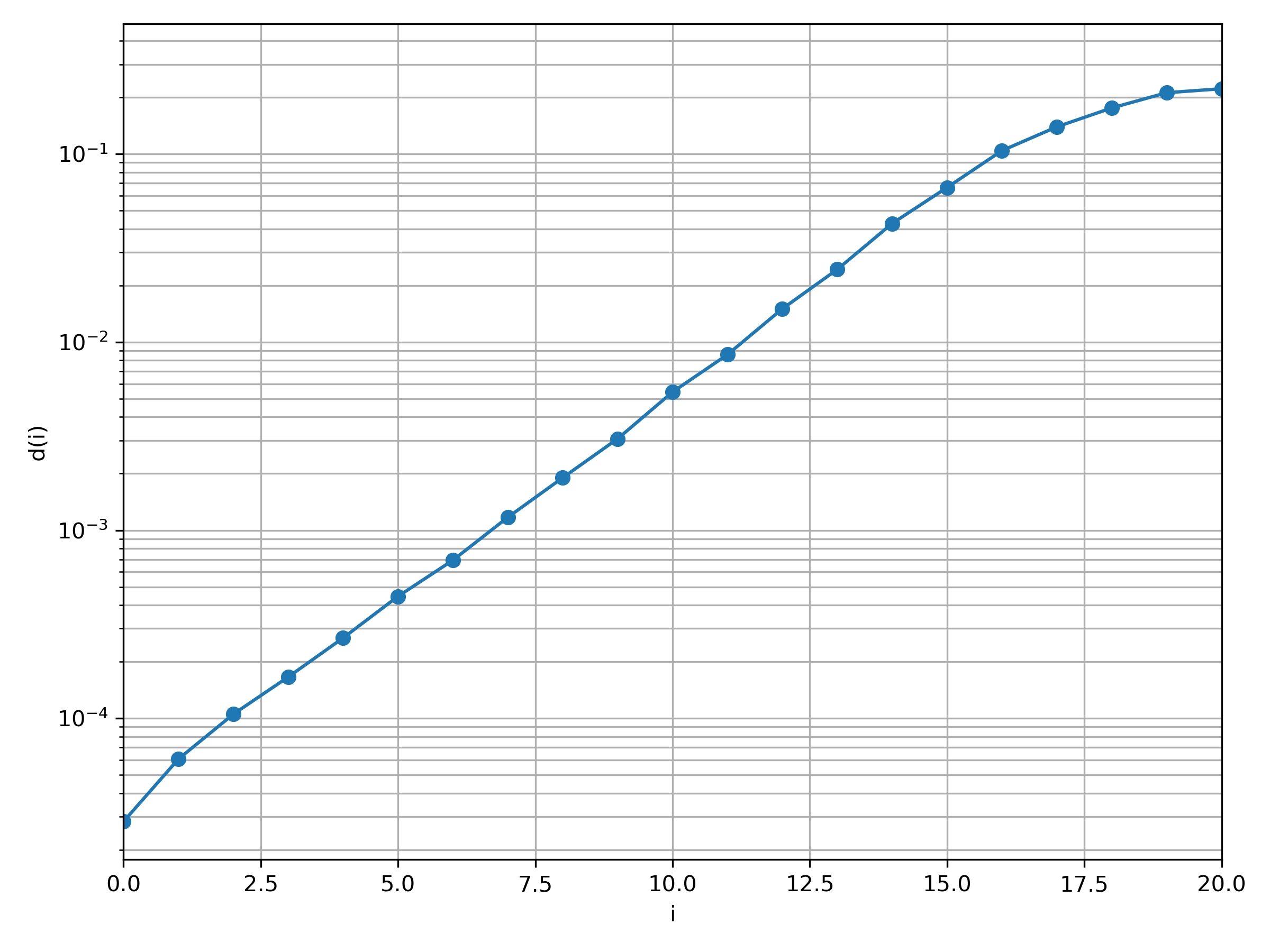}
	\caption{Mean separation $d(j)$ of reference and perturbed logistic-map trajectories after an initial evolution of $10^5$ steps, for $h(i)=\lfloor\sqrt{i}/4\rfloor$ and $\epsilon=0.05$. The plot covers 20 further updates and averages over $10^3$ configurations, showing approximately exponential separation initially and subsequent saturation.}
	\label{lyap2}
\end{figure}

\clearpage
\section{Conclusions, Discussion, and Future Directions} \label{sec4}

In this paper, we investigated pantograph difference equations in which the next iterate depends on both the current state and a state evaluated at a proportional or sublinear proportional index. The linear equation
\[
x_{n+1}=\lambda x_n+\delta x_{h(n)}
\]
provides the basic stability framework. When $h(n)\to\infty$, Theorem \ref{thm1} shows that $|\lambda|+|\delta|<1$ is a delay-independent sufficient condition for asymptotic stability of the zero solution.

Theorems \ref{thm2}--\ref{thm5} provide complementary instability and growth results. Theorem \ref{thm2} shows that, for $\lambda,\delta\geq0$ with $\lambda+\delta>1$, the zero solution is not asymptotically stable for any delay satisfying $h(n)\leq n$. Theorem \ref{thm5} strengthens this conclusion when $\lambda,\delta>0$ and $h$ is non-decreasing: every solution with $x_0>0$ is strictly increasing; for the proportional delay $h(n)=\lfloor an\rfloor$, it is consequently unbounded. Theorem \ref{thm3} proves that $|\lambda|>1+|\delta|$ forces every nonzero solution to be unbounded, independently of the particular delay. In the complementary dominance region, Theorem \ref{thm4} shows that, for $h(n)=\lfloor an\rfloor$, the condition $|\delta|>1+|\lambda|$ forces every nonzero solution to possess an unbounded subsequence. These results rigorously identify stable and unstable regions in the $(\lambda,\delta)$-plane, while leaving intermediate regions in which the outcome can depend on the proportional-delay parameter $a$.

The analytical results were supplemented by parameter scans for $a=0.1,0.2,\ldots,0.9$. Multi-horizon envelope and RMS diagnostics showed that many trajectories appearing bounded at a single terminal time actually exhibit ultra-slow growth at longer horizons. The computations also indicate that, within the delay-dependent region, the numerically stable portion generally shrinks as $a$ increases. The resulting regime diagrams should therefore be interpreted as a combination of rigorous analytical regions and finite-horizon numerical evidence, not as a complete proof of all transition boundaries.

For nonlinear delayed maps, linearization reduces local stability to the corresponding linear pantograph equation. This yields delay-independent sufficient conditions for fixed-point stability and applies directly to the delayed logistic map. In the homogeneous logistic and tent maps, proportional delay organizes the orbit into long plateaus or stasis intervals. Their lengths grow geometrically for $h(i)=\lfloor i/k\rfloor$ and even faster for the sublinear delays considered here. Consequently, asymptotic behavior may remain computationally inaccessible, producing pronounced slow convergence and finite-time multistability.

In the inhomogeneous case, numerical experiments show finite-time exponential separation of nearby trajectories, consistent with chaotic dynamics, although they do not establish an asymptotic Lyapunov exponent. Further work is needed to obtain sharp transition curves in the delay-dependent regions, rigorous asymptotic growth rates, and Lyapunov-type results for nonstationary delays. Extensions to higher-dimensional maps, coupled map lattices, and random or state-dependent proportional delays also provide natural directions for future research.

	
	
	%
	
	


\section*{Appendix}
\small
We provide some terms in the solution of the delay difference equation (\ref{pmap}) with $h(n)=\lfloor n/m \rfloor,$ $m=2,3,4,5$.

\bigskip
\noindent\textbf{The solution terms for $h(n)=\lfloor n/2 \rfloor$}:

\noindent\(x_1=\delta +\lambda\)

\noindent\(x_2=\delta +\delta  \lambda +\lambda ^2\)

\noindent\(x_3=\delta ^2+2 \delta  \lambda +\delta  \lambda ^2+\lambda ^3\)

\noindent\(x_4=\delta ^2+\left(\delta +\delta ^2\right) \lambda +2 \delta  \lambda ^2+\delta  \lambda ^3+\lambda ^4\)

\noindent\(x_5=\delta ^2+2 \delta ^2 \lambda +\left(2 \delta +\delta ^2\right) \lambda ^2+2 \delta  \lambda ^3+\delta  \lambda ^4+\lambda ^5\)

\noindent\(x_6=\delta ^2+2 \delta ^2 \lambda +\left(\delta +2 \delta ^2\right) \lambda ^2+\left(2 \delta +\delta ^2\right) \lambda ^3+2 \delta  \lambda
^4+\delta  \lambda ^5+\lambda ^6\)

\noindent\(x_7=\delta ^3+3 \delta ^2 \lambda +3 \delta ^2 \lambda ^2+\left(2 \delta +2 \delta ^2\right) \lambda ^3+\left(2 \delta +\delta ^2\right)
\lambda ^4+2 \delta  \lambda ^5+\delta  \lambda ^6+\lambda ^7\)

\noindent\(x_8=\delta ^3+\left(2 \delta ^2+\delta ^3\right) \lambda +4 \delta ^2 \lambda ^2+\left(\delta +3 \delta ^2\right) \lambda ^3+\left(2 \delta
+2 \delta ^2\right) \lambda ^4+\left(2 \delta +\delta ^2\right) \lambda ^5+2 \delta  \lambda ^6+\delta  \lambda ^7+\lambda ^8\)

\noindent\(x_9=\delta ^3+\left(\delta ^2+2 \delta ^3\right) \lambda +\left(4 \delta ^2+\delta ^3\right) \lambda ^2+5 \delta ^2 \lambda ^3+\left(2
\delta +3 \delta ^2\right) \lambda ^4+\left(2 \delta +2 \delta ^2\right) \lambda ^5+\left(2 \delta +\delta ^2\right) \lambda ^6+2 \delta  \lambda
^7+\delta  \lambda ^8+\lambda ^9\)

\noindent\(x_{10}=\delta ^3+\left(\delta ^2+2 \delta ^3\right) \lambda +\left(3 \delta ^2+2 \delta ^3\right) \lambda ^2+\left(5 \delta ^2+\delta
^3\right) \lambda ^3+\left(\delta +5 \delta ^2\right) \lambda ^4+\left(2 \delta +3 \delta ^2\right) \lambda ^5+\left(2 \delta +2 \delta ^2\right)
\lambda ^6+\left(2 \delta +\delta ^2\right) \lambda ^7+2 \delta  \lambda ^8+\delta  \lambda ^9+\lambda ^{10}\)

\noindent\(x_{11}=\delta ^3+3 \delta ^3 \lambda +\left(3 \delta ^2+3 \delta ^3\right) \lambda ^2+\left(5 \delta ^2+2 \delta ^3\right) \lambda ^3+\left(6
\delta ^2+\delta ^3\right) \lambda ^4+\left(2 \delta +5 \delta ^2\right) \lambda ^5+\left(2 \delta +3 \delta ^2\right) \lambda ^6+\left(2 \delta
+2 \delta ^2\right) \lambda ^7+\left(2 \delta +\delta ^2\right) \lambda ^8+2 \delta  \lambda ^9+\delta  \lambda ^{10}+\lambda ^{11}\)

\noindent\(x_{12}=\delta ^3+3 \delta ^3 \lambda +\left(2 \delta ^2+4 \delta ^3\right) \lambda ^2+\left(5 \delta ^2+3 \delta ^3\right) \lambda ^3+\left(6
\delta ^2+2 \delta ^3\right) \lambda ^4+\left(\delta +6 \delta ^2+\delta ^3\right) \lambda ^5+\left(2 \delta +5 \delta ^2\right) \lambda ^6+\left(2
\delta +3 \delta ^2\right) \lambda ^7+\left(2 \delta +2 \delta ^2\right) \lambda ^8+\left(2 \delta +\delta ^2\right) \lambda ^9+2 \delta  \lambda
^{10}+\delta  \lambda ^{11}+\lambda ^{12}\)

\noindent\(x_{13}=\delta ^3+3 \delta ^3 \lambda +\left(\delta ^2+5 \delta ^3\right) \lambda ^2+\left(4 \delta ^2+5 \delta ^3\right) \lambda ^3+\left(7
\delta ^2+3 \delta ^3\right) \lambda ^4+\left(7 \delta ^2+2 \delta ^3\right) \lambda ^5+\left(2 \delta +6 \delta ^2+\delta ^3\right) \lambda ^6+\left(2
\delta +5 \delta ^2\right) \lambda ^7+\left(2 \delta +3 \delta ^2\right) \lambda ^8+\left(2 \delta +2 \delta ^2\right) \lambda ^9+\left(2 \delta
+\delta ^2\right) \lambda ^{10}+2 \delta  \lambda ^{11}+\delta  \lambda ^{12}+\lambda ^{13}\)

\bigskip
\noindent\textbf{The solution terms for $h(n)=\lfloor n/3 \rfloor$}:

\noindent\(x_1=\delta +\lambda\)

\noindent\(x_2=\delta +\delta  \lambda +\lambda ^2\)

\noindent\(x_3=\delta +\delta  \lambda +\delta  \lambda ^2+\lambda ^3\)

\noindent\(x_4=\delta ^2+2 \delta  \lambda +\delta  \lambda ^2+\delta  \lambda ^3+\lambda ^4\)

\noindent\(x_5=\delta ^2+\left(\delta +\delta ^2\right) \lambda +2 \delta  \lambda ^2+\delta  \lambda ^3+\delta  \lambda ^4+\lambda ^5\)

\noindent\(x_6=\delta ^2+\left(\delta +\delta ^2\right) \lambda +\left(\delta +\delta ^2\right) \lambda ^2+2 \delta  \lambda ^3+\delta  \lambda ^4+\delta
 \lambda ^5+\lambda ^6\)

\noindent\(x_7=\delta ^2+2 \delta ^2 \lambda +\left(2 \delta +\delta ^2\right) \lambda ^2+\left(\delta +\delta ^2\right) \lambda ^3+2 \delta  \lambda
^4+\delta  \lambda ^5+\delta  \lambda ^6+\lambda ^7\)

\noindent\(x_8=\delta ^2+2 \delta ^2 \lambda +\left(\delta +2 \delta ^2\right) \lambda ^2+\left(2 \delta +\delta ^2\right) \lambda ^3+\left(\delta
+\delta ^2\right) \lambda ^4+2 \delta  \lambda ^5+\delta  \lambda ^6+\delta  \lambda ^7+\lambda ^8\)

\noindent\(x_9=\delta ^2+2 \delta ^2 \lambda +\left(\delta +2 \delta ^2\right) \lambda ^2+\left(\delta +2 \delta ^2\right) \lambda ^3+\left(2 \delta
+\delta ^2\right) \lambda ^4+\left(\delta +\delta ^2\right) \lambda ^5+2 \delta  \lambda ^6+\delta  \lambda ^7+\delta  \lambda ^8+\lambda ^9\)

\noindent\(x_{10}=\delta ^2+2 \delta ^2 \lambda +3 \delta ^2 \lambda ^2+\left(2 \delta +2 \delta ^2\right) \lambda ^3+\left(\delta +2 \delta ^2\right)
\lambda ^4+\left(2 \delta +\delta ^2\right) \lambda ^5+\left(\delta +\delta ^2\right) \lambda ^6+2 \delta  \lambda ^7+\delta  \lambda ^8+\delta 
\lambda ^9+\lambda ^{10}\)

\noindent\(x_{11}=\delta ^2+2 \delta ^2 \lambda +3 \delta ^2 \lambda ^2+\left(\delta +3 \delta ^2\right) \lambda ^3+\left(2 \delta +2 \delta ^2\right)
\lambda ^4+\left(\delta +2 \delta ^2\right) \lambda ^5+\left(2 \delta +\delta ^2\right) \lambda ^6+\left(\delta +\delta ^2\right) \lambda ^7+2 \delta
 \lambda ^8+\delta  \lambda ^9+\delta  \lambda ^{10}+\lambda ^{11}\)

\noindent\(x_{12}=\delta ^2+2 \delta ^2 \lambda +3 \delta ^2 \lambda ^2+\left(\delta +3 \delta ^2\right) \lambda ^3+\left(\delta +3 \delta ^2\right)
\lambda ^4+\left(2 \delta +2 \delta ^2\right) \lambda ^5+\left(\delta +2 \delta ^2\right) \lambda ^6+\left(2 \delta +\delta ^2\right) \lambda ^7+\left(\delta
+\delta ^2\right) \lambda ^8+2 \delta  \lambda ^9+\delta  \lambda ^{10}+\delta  \lambda ^{11}+\lambda ^{12}\)

\noindent\(x_{13}=\delta ^3+3 \delta ^2 \lambda +3 \delta ^2 \lambda ^2+4 \delta ^2 \lambda ^3+\left(2 \delta +3 \delta ^2\right) \lambda ^4+\left(\delta
+3 \delta ^2\right) \lambda ^5+\left(2 \delta +2 \delta ^2\right) \lambda ^6+\left(\delta +2 \delta ^2\right) \lambda ^7+\left(2 \delta +\delta ^2\right)
\lambda ^8+\left(\delta +\delta ^2\right) \lambda ^9+2 \delta  \lambda ^{10}+\delta  \lambda ^{11}+\delta  \lambda ^{12}+\lambda ^{13}\)

\bigskip
\noindent\textbf{The solution terms for $h(n)=\lfloor n/4 \rfloor$}:

\noindent\(x_1=\delta +\lambda\)

\noindent\(x_2=\delta +\delta  \lambda +\lambda ^2\)

\noindent\(x_3=\delta +\delta  \lambda +\delta  \lambda ^2+\lambda ^3\)

\noindent\(x_4=\delta +\delta  \lambda +\delta  \lambda ^2+\delta  \lambda ^3+\lambda ^4\)

\noindent\(x_5=\delta ^2+2 \delta  \lambda +\delta  \lambda ^2+\delta  \lambda ^3+\delta  \lambda ^4+\lambda ^5\)

\noindent\(x_6=\delta ^2+\left(\delta +\delta ^2\right) \lambda +2 \delta  \lambda ^2+\delta  \lambda ^3+\delta  \lambda ^4+\delta  \lambda ^5+\lambda
^6\)

\noindent\(x_7=\delta ^2+\left(\delta +\delta ^2\right) \lambda +\left(\delta +\delta ^2\right) \lambda ^2+2 \delta  \lambda ^3+\delta  \lambda ^4+\delta
 \lambda ^5+\delta  \lambda ^6+\lambda ^7\)

\noindent\(x_8=\delta ^2+\left(\delta +\delta ^2\right) \lambda +\left(\delta +\delta ^2\right) \lambda ^2+\left(\delta +\delta ^2\right) \lambda
^3+2 \delta  \lambda ^4+\delta  \lambda ^5+\delta  \lambda ^6+\delta  \lambda ^7+\lambda ^8\)

\noindent\(x_9=\delta ^2+2 \delta ^2 \lambda +\left(2 \delta +\delta ^2\right) \lambda ^2+\left(\delta +\delta ^2\right) \lambda ^3+\left(\delta
+\delta ^2\right) \lambda ^4+2 \delta  \lambda ^5+\delta  \lambda ^6+\delta  \lambda ^7+\delta  \lambda ^8+\lambda ^9\)

\noindent\(x_{10}=\delta ^2+2 \delta ^2 \lambda +\left(\delta +2 \delta ^2\right) \lambda ^2+\left(2 \delta +\delta ^2\right) \lambda ^3+\left(\delta
+\delta ^2\right) \lambda ^4+\left(\delta +\delta ^2\right) \lambda ^5+2 \delta  \lambda ^6+\delta  \lambda ^7+\delta  \lambda ^8+\delta  \lambda
^9+\lambda ^{10}\)

\noindent\(x_{11}=\delta ^2+2 \delta ^2 \lambda +\left(\delta +2 \delta ^2\right) \lambda ^2+\left(\delta +2 \delta ^2\right) \lambda ^3+\left(2
\delta +\delta ^2\right) \lambda ^4+\left(\delta +\delta ^2\right) \lambda ^5+\left(\delta +\delta ^2\right) \lambda ^6+2 \delta  \lambda ^7+\delta
 \lambda ^8+\delta  \lambda ^9+\delta  \lambda ^{10}+\lambda ^{11}\)

\noindent\(x_{12}=\delta ^2+2 \delta ^2 \lambda +\left(\delta +2 \delta ^2\right) \lambda ^2+\left(\delta +2 \delta ^2\right) \lambda ^3+\left(\delta
+2 \delta ^2\right) \lambda ^4+\left(2 \delta +\delta ^2\right) \lambda ^5+\left(\delta +\delta ^2\right) \lambda ^6+\left(\delta +\delta ^2\right)
\lambda ^7+2 \delta  \lambda ^8+\delta  \lambda ^9+\delta  \lambda ^{10}+\delta  \lambda ^{11}+\lambda ^{12}\)

\noindent\(x_{13}=\delta ^2+2 \delta ^2 \lambda +3 \delta ^2 \lambda ^2+\left(2 \delta +2 \delta ^2\right) \lambda ^3+\left(\delta +2 \delta ^2\right)
\lambda ^4+\left(\delta +2 \delta ^2\right) \lambda ^5+\left(2 \delta +\delta ^2\right) \lambda ^6+\left(\delta +\delta ^2\right) \lambda ^7+\left(\delta
+\delta ^2\right) \lambda ^8+2 \delta  \lambda ^9+\delta  \lambda ^{10}+\delta  \lambda ^{11}+\delta  \lambda ^{12}+\lambda ^{13}\)

\bigskip
\noindent\textbf{The solution terms for $h(n)=\lfloor n/5 \rfloor$}:

\noindent\(x_1=\delta +\lambda\)

\noindent\(x_2=\delta +\delta  \lambda +\lambda ^2\)

\noindent\(x_3=\delta +\delta  \lambda +\delta  \lambda ^2+\lambda ^3\)

\noindent\(x_4=\delta +\delta  \lambda +\delta  \lambda ^2+\delta  \lambda ^3+\lambda ^4\)

\noindent\(x_5=\delta +\delta  \lambda +\delta  \lambda ^2+\delta  \lambda ^3+\delta  \lambda ^4+\lambda ^5\)

\noindent\(x_6=\delta ^2+2 \delta  \lambda +\delta  \lambda ^2+\delta  \lambda ^3+\delta  \lambda ^4+\delta  \lambda ^5+\lambda ^6\)

\noindent\(x_7=\delta ^2+\left(\delta +\delta ^2\right) \lambda +2 \delta  \lambda ^2+\delta  \lambda ^3+\delta  \lambda ^4+\delta  \lambda ^5+\delta
 \lambda ^6+\lambda ^7\)

\noindent\(x_8=\delta ^2+\left(\delta +\delta ^2\right) \lambda +\left(\delta +\delta ^2\right) \lambda ^2+2 \delta  \lambda ^3+\delta  \lambda ^4+\delta
 \lambda ^5+\delta  \lambda ^6+\delta  \lambda ^7+\lambda ^8\)

\noindent\(x_9=\delta ^2+\left(\delta +\delta ^2\right) \lambda +\left(\delta +\delta ^2\right) \lambda ^2+\left(\delta +\delta ^2\right) \lambda
^3+2 \delta  \lambda ^4+\delta  \lambda ^5+\delta  \lambda ^6+\delta  \lambda ^7+\delta  \lambda ^8+\lambda ^9\)

\noindent\(x_{10}=\delta ^2+\left(\delta +\delta ^2\right) \lambda +\left(\delta +\delta ^2\right) \lambda ^2+\left(\delta +\delta ^2\right) \lambda
^3+\left(\delta +\delta ^2\right) \lambda ^4+2 \delta  \lambda ^5+\delta  \lambda ^6+\delta  \lambda ^7+\delta  \lambda ^8+\delta  \lambda ^9+\lambda
^{10}\)

\noindent\(x_{11}=\delta ^2+2 \delta ^2 \lambda +\left(2 \delta +\delta ^2\right) \lambda ^2+\left(\delta +\delta ^2\right) \lambda ^3+\left(\delta
+\delta ^2\right) \lambda ^4+\left(\delta +\delta ^2\right) \lambda ^5+2 \delta  \lambda ^6+\delta  \lambda ^7+\delta  \lambda ^8+\delta  \lambda
^9+\delta  \lambda ^{10}+\lambda ^{11}\)

\noindent\(x_{12}=\delta ^2+2 \delta ^2 \lambda +\left(\delta +2 \delta ^2\right) \lambda ^2+\left(2 \delta +\delta ^2\right) \lambda ^3+\left(\delta
+\delta ^2\right) \lambda ^4+\left(\delta +\delta ^2\right) \lambda ^5+\left(\delta +\delta ^2\right) \lambda ^6+2 \delta  \lambda ^7+\delta  \lambda
^8+\delta  \lambda ^9+\delta  \lambda ^{10}+\delta  \lambda ^{11}+\lambda ^{12}\)

\noindent\(x_{13}=\delta ^2+2 \delta ^2 \lambda +\left(\delta +2 \delta ^2\right) \lambda ^2+\left(\delta +2 \delta ^2\right) \lambda ^3+\left(2
\delta +\delta ^2\right) \lambda ^4+\left(\delta +\delta ^2\right) \lambda ^5+\left(\delta +\delta ^2\right) \lambda ^6+\left(\delta +\delta ^2\right)
\lambda ^7+2 \delta  \lambda ^8+\delta  \lambda ^9+\delta  \lambda ^{10}+\delta  \lambda ^{11}+\delta  \lambda ^{12}+\lambda ^{13}\)

\normalsize
\section*{Declarations}
The authors declare that they have no known competing financial interests or personal relationships that could have appeared to influence the work reported in this paper.\\
Generative AI tools were used only to improve the language and 
clarity of this manuscript. The authors carefully reviewed and 
edited the output and take full responsibility for the final content.

\bibliographystyle{plain}      
\bibliography{ref_revised.bib}

@book{elaydi2005introduction,
  title={An introduction to difference equations},
  author={Elaydi, Saber},
  year={2005},
  publisher={Springer}
}

@article{ockendon1971dynamics,
  title={The dynamics of a current collection system for an electric locomotive},
  author={Ockendon, J. R. and Tayler, A. B.},
  journal={Proceedings of the Royal Society of London. A. Mathematical and Physical Sciences},
  volume={322},
  number={1551},
  pages={447--468},
  year={1971},
  publisher={The Royal Society},
  doi={10.1098/rspa.1971.0078}
}

@article{kato1971functional,
  title={The functional-differential equation $y'(x)=ay(\lambda x)+by(x)$},
  author={Kato, Tosio and McLeod, J. B.},
  journal={Bulletin of the American Mathematical Society},
  volume={77},
  number={6},
  pages={891--937},
  year={1971}
}

@article{iserles1993generalized,
  title={On the generalized pantograph functional-differential equation},
  author={Iserles, A.},
  journal={European Journal of Applied Mathematics},
  volume={4},
  number={1},
  pages={1--38},
  year={1993},
  doi={10.1017/S0956792500000966}
}

@article{buhmann1993stability,
  title={Stability of the discretized pantograph differential equation},
  author={Buhmann, Martin D. and Iserles, Arieh},
  journal={Mathematics of Computation},
  volume={60},
  number={202},
  pages={575--589},
  year={1993},
  doi={10.1090/S0025-5718-1993-1176707-2}
}

@article{peics2000asymptotic,
  title={On the asymptotic behaviour of a pantograph-type difference equation},
  author={P{\'e}ics, Hajnalka},
  journal={Journal of Difference Equations and Applications},
  volume={6},
  number={3},
  pages={257--273},
  year={2000},
  doi={10.1080/10236190008808228}
}

@article{kundrat2006asymptotics,
  title={On the asymptotics of the difference equation with a proportional delay},
  author={Kundr{\'a}t, Petr},
  journal={Opuscula Mathematica},
  volume={26},
  number={3},
  pages={499--506},
  year={2006}
}

@book{bellen2003numerical,
  title={Numerical Methods for Delay Differential Equations},
  author={Bellen, Alfredo and Zennaro, Marino},
  series={Numerical Mathematics and Scientific Computation},
  publisher={Oxford University Press},
  address={New York},
  year={2003},
  doi={10.1093/acprof:oso/9780198506546.001.0001}
}

@article{jafari2021new,
  title={A new numerical method to solve pantograph delay differential equations with convergence analysis},
  author={Jafari, H. and Mahmoudi, M. and Noori Skandari, M. H.},
  journal={Advances in Difference Equations},
  volume={2021},
  number={129},
  year={2021},
  doi={10.1186/s13662-021-03293-0}
}

@article{rihan2022theta,
  title={An analysis of the theta-method for pantograph-type delay differential equations},
  author={Rihan, Fathalla A. and Rihan, Ahmed F.},
  journal={Complexity},
  volume={2022},
  pages={8961352},
  year={2022},
  doi={10.1155/2022/8961352}
}

@article{albidah2023exact,
  title={Exact and numerical analysis of the pantograph delay differential equation via the homotopy perturbation method},
  author={Albidah, Abdulrahman B. and Kanaan, Nourah E. and Ebaid, Abdelhalim and Al-Jeaid, Hind K.},
  journal={Mathematics},
  volume={11},
  number={4},
  pages={944},
  year={2023},
  doi={10.3390/math11040944}
}

@article{rihan2024runge,
  title={Continuous Runge--Kutta schemes for pantograph type delay differential equations},
  author={Rihan, Fathalla A.},
  journal={Partial Differential Equations in Applied Mathematics},
  volume={11},
  pages={100797},
  year={2024},
  doi={10.1016/j.padiff.2024.100797}
}

@article{aljohani2024analytical,
  title={Analytical and numerical investigation for the inhomogeneous pantograph equation},
  author={Aljohani, A. F. and Ebaid, Abdelhalim and others},
  journal={Axioms},
  volume={13},
  number={6},
  pages={377},
  year={2024},
  doi={10.3390/axioms13060377}
}

@article{abusalim2024some,
  title={Some results for a class of pantograph integro-fractional stochastic differential equations},
  author={Abusalim, Sahar Mohammad and Fakhfakh, Raouf and Alshahrani, Fatimah and Ben Makhlouf, Abdellatif},
  journal={Symmetry},
  volume={16},
  number={10},
  pages={1362},
  year={2024},
  doi={10.3390/sym16101362}
}

@article{bhalekar2016analytical,
  title={Analytical solutions of nonlinear equations with proportional delays},
  author={Bhalekar, Sachin and Patade, Jayvant},
  journal={Applied and Computational Mathematics},
  volume={15},
  number={3},
  pages={331--345},
  year={2016}
}

@article{bhalekar2017series,
  title={Series solution of the pantograph equation and its properties},
  author={Bhalekar, Sachin and Patade, Jayvant},
  journal={Fractal and Fractional},
  volume={1},
  number={1},
  pages={16},
  year={2017},
  doi={10.3390/fractalfract1010016}
}

@article{patade2017analytical,
  title={Analytical solution of pantograph equation with incommensurate delay},
  author={Patade, Jayvant and Bhalekar, Sachin},
  journal={Physical Sciences Reviews},
  volume={2},
  number={9},
  pages={20165103},
  year={2017},
  doi={10.1515/psr-2016-5103}
}

@article{bhalekar2026stability,
  title={Stability Analysis of Pantograph Delay Differential Equations},
  author={Bhalekar, Sachin},
  journal={arXiv preprint arXiv:2605.22019},
  year={2026}
}

@article{bahsi2015numerical,
  title={Numerical solution of pantograph-type delay differential equations using perturbation-iteration algorithms},
  author={Bah{\c{s}}i, M. Mustafa and {\c{C}}evik, Mehmet},
  journal={Journal of Applied Mathematics},
  volume={2015},
  number={1},
  pages={139821},
  year={2015}
}

@article{ahmad2023controllability,
  title={Controllability of a generalized multi-pantograph system of non-integer order with state delay},
  author={Ahmad, Irshad and others},
  journal={AIMS Mathematics},
  volume={8},
  number={6},
  pages={13764--13784},
  year={2023}
}

@article{alrebdi2023accurate,
  title={Accurate solution for the pantograph delay differential equation via Laplace transform},
  author={Alrebdi, Reem and Al-Jeaid, Hind K.},
  journal={Mathematics},
  volume={11},
  number={9},
  pages={2031},
  year={2023}
}

@article{yi2025novel,
  title={A novel successive convexification framework with symplectic pseudospectral solutions for nonlinear optimal control problems of constrained pantograph delay systems},
  author={Yi, Xueling and others},
  journal={International Journal of Systems Science},
  volume={56},
  number={2},
  pages={237--266},
  year={2025}
}

@article{jamaludin2025adaptive,
  title={Adaptive predictor-corrector schemes for a class of pantograph-type delay differential equations},
  author={Jamaludin, Nur Aishah and Abdul Aziz, Nurul Huda},
  journal={Journal of Quality Measurement and Analysis},
  volume={21},
  number={4},
  pages={313--327},
  year={2025}
}

@article{bruni2025recent,
  title={Recent developments on pantograph-overhead line interaction},
  author={Bruni, Stefano and others},
  journal={Vehicle System Dynamics},
  volume={63},
  number={7},
  pages={1358--1394},
  year={2025}
}

@article{mohammad2026carleson,
  title={Carleson frame methods for pantograph-type delay differential equations: Theory and applications},
  author={Mohammad, Mutaz and Sweidan, Mohyeedden and Trounev, Alexander},
  journal={Mathematical Methods in the Applied Sciences},
  year={2026}
}

@book{bendat2010random,
  title={Random Data: Analysis and Measurement Procedures},
  author={Bendat, Julius S. and Piersol, Allan G.},
  edition={4},
  publisher={John Wiley \& Sons},
  address={Hoboken, NJ},
  year={2010},
  doi={10.1002/9781118032428}
}

@book{press2007numerical,
  title={Numerical Recipes: The Art of Scientific Computing},
  author={Press, William H. and Teukolsky, Saul A. and Vetterling, William T. and Flannery, Brian P.},
  edition={3},
  publisher={Cambridge University Press},
  address={Cambridge},
  year={2007}
}

@article{clauset2009powerlaw,
  title={Power-Law Distributions in Empirical Data},
  author={Clauset, Aaron and Shalizi, Cosma Rohilla and Newman, M. E. J.},
  journal={SIAM Review},
  volume={51},
  number={4},
  pages={661--703},
  year={2009},
  doi={10.1137/070710111}
}

\end{document}